\documentclass[11pt, oneside]{amsart}   	% 
\usepackage{geometry,amsmath,amssymb,verbatim,tikz,ytableau,hyperref}
\usepackage[shortlabels]{enumitem}
\usepackage[utf8]{inputenc}
\usepackage{quiver}
\usepackage[mathscr]{euscript}
\usepackage{mathdots}
\usepackage{genyoungtabtikz}
\usepackage[normalem]{ulem}
\usepackage{shuffle}
\usepackage{centernot}
\usepackage{mathtools}
\usepackage{ stmaryrd }

\newtheorem{thm}{Theorem}[section]
\newtheorem{lemma}[thm]{Lemma}

\newtheorem{prop}[thm]{Proposition}
\newtheorem{cor}[thm]{Corollary}

\newtheorem{question}[thm]{Question}
\newtheorem{definition}[thm]{Definition}

\theoremstyle{definition}
\newtheorem{example}[thm]{Example}

\newtheorem{remark}[thm]{Remark}

\def\bbs{{\mathbb{S}}}
\def\yshort{\ytableaushort}
\def\rev{\operatorname{rev}}
\def\trace{\operatorname{trace}}
\def\rank{\operatorname{rank}}

\def\diag{\operatorname{diag}}
\def\rev{\operatorname{rev}}

\def\gl{GL}%%changed this to be consistent throughout

\def\ZZ{{\mathbb{Z}}}
\def\NN{{\mathbb{N}}}
\def\schurfunctor{{\mathbb{S}}}

\newcommand{\kk}{\mathbf k}
\renewcommand{\aa}{\mathbf a}
\newcommand{\bb}{\mathbf b}
\newcommand{\cc}{\mathbf c}

\newcommand{\pp}{\mathbf p}
\newcommand{\qq}{\mathbf q}

\newcommand{\xx}{\mathbf x}

\def\Ext{\operatorname{Ext}}
\def\ext{\operatorname{Ext}}
\def\Tor{\operatorname{Tor}}
\def\tor{\operatorname{Tor}}
\def\Hom{\operatorname{Hom}}
\def\ch{\operatorname{ch}}
\def\coker{\operatorname{coker}}
\def\im{\operatorname{im}}
\def\rows{\operatorname{rows}}
\def\cols{\operatorname{cols}}

\def\specht{\mathcal{S}}
\def\PP{{\mathcal{P}}}
\def\HHH{{\mathcal{H}}}

\def\sgn{\operatorname{sgn}}

\def\cone{\operatorname{cone}}

\def\kk{{\mathbf{k}}}

\def\symm{{\mathfrak{S}}}

\def\fancyR{\mathscr{R}}
\def\fancyA{\mathscr{A}}
\def\fancyB{\mathscr{B}}

\def\xx{{\mathbf{x}}}

\def\aa{{\mathbf{a}}}
\def\bb{{\mathbf{b}}}

\DeclarePairedDelimiter\abs{\lvert}{\rvert}%

\newcommand{\ra}{\rightarrow}
\renewcommand{\hom}{\operatorname{Hom}}
\newcommand{\veralg}[1]{{S^{(#1)}}}

\newcommand{\veralgplus}[1]{{S_+^{(#1)}}}
\newcommand{\vermod}[2]{{S^{({\geq #2},{#1})}}}

\title[Equivariant resolutions over Veronese rings]{Equivariant resolutions over Veronese rings}

\author{Ayah Almousa}
\author{Michael Perlman}
\author{Alexandra Pevzner}
\author{Victor Reiner}
\email{almou007@umn.edu, mperlman@umn.edu, pevzn002@umn.edu, reiner@umn.edu}
\address{School of Mathematics, University of Minnesota, Minneapolis, MN 55455}
\author{Keller VandeBogert}
\email{kvandebo@nd.edu}
\address{Department of Mathematics, Notre Dame, IN 46556 }


\begin{document}

\begin{abstract}
Working in a polynomial ring $S=\kk[x_1,\ldots,x_n]$
where $\kk$ is an arbitrary commutative ring with $1$, we consider the $d^{th}$ Veronese subalgebras $R=\veralg{d}$, as well as natural $R$-submodules $M=\vermod{d}{r}$ inside $S$.
We develop and use characteristic-free theory of Schur functors associated to ribbon skew diagrams as a tool to construct simple $GL_n(\kk)$-equivariant 
minimal free $R$-resolutions for the quotient ring $\kk=R/R_+$
and for these modules $M$.  These also lead to
elegant descriptions of $\Tor^R_i(M,M')$ for all $i$ and $\hom_R(M,M')$ for any pair of these modules $M,M'$.
\end{abstract}

\maketitle
%\tableofcontents

%%%%%%%%%%%%%%%%%%%%%%%%%%%%
\section{Introduction}
\label{intro-section}
%%%%%%%%%%%%%%%%%%%%%%%%%%%%

Let $S=\kk[x_1,\ldots,x_n]$ be a polynomial ring 
where $\kk$ is a commutative ring with unit.
Consider $S$ as a standard graded $\kk$-algebra $S=\oplus_j S_j$ in which $S_j$ are the homogeneous polynomials of degree $j$.  Then $S$ contains for each $d=1,2,3,\ldots$ 
a {\it Veronese subalgebra}
$$
R=\veralg{d}:=
\bigoplus_{j \equiv 0 \bmod{d}} S_j.
$$
Fixing $d$, our goal is
to study various $R$-modules $M$ inside $S$ of the form 
$$
M=\vermod{d}{ r}:=
\bigoplus_{\substack{j \geq r\\j \equiv r \bmod{d}}} S_j
=
S_r \oplus S_{r+d} \oplus S_{r+2d}
\oplus S_{r+3d} \oplus \cdots
$$
where $r=0,1,2,\ldots.$
Our first main result, Theorem~\ref{skew-Schur-functor-theorem} below, describes a (generally infinite) minimal\footnote{We do not assume $\kk$ is a field, so ``minimal" here  means the differentials have all entries in $R_+=\vermod{d}{d}$.} $R$-free resolution for each such $M$ as an $R$-module. When $\kk$ is a field, it is known (see Section~\ref{tor-koszulity-subsection} below) that $M$ always has a {\it linear} $R$-free resolution\footnote{That is, after rescaling the grading of $R=\veralg{d}$ to make its generators of degree $1$ (not degree $d$) and shifting the grading on $M=\vermod{d}{ r}$
to make its $R$-module generators of degree $0$ (not degree $r$).}
\begin{equation}
    \label{generic-MFR}
0 \leftarrow M 
\leftarrow R^{\beta_0}
\leftarrow R(-1)^{\beta_1}
\leftarrow R(-2)^{\beta_2}
\leftarrow \cdots
\end{equation}
However, we will aim for more precise {\it equivariant} descriptions using the action of the 
{\it general linear group} $GL(V)$ on the {\it symmetric algebra}
$$
S=\kk[x_1,\ldots,x_n]=S(V),
$$
in which $V=\kk^n$ is a free $\kk$-module
with $\kk$-basis $x_1,\ldots,x_n$, so that $S_j \cong S^j(V)$, the $j^{th}$ symmetric power of $V$.  We should note that, by contrast, there has been more extensive study of equivariant descriptions of
the {\it finite} free resolutions for both $R$ itself and the modules $M$ over the {\it polynomial} ring $S(S^d(V))$; see \cite{bruce2020conjectures, GrecoMartino,JPWmatrices,  OttavianiPaoletti,reiner2000minimal, sam2014derived}

Our results use polynomial representations of $GL(V)$, including the {\it Schur functors} $V \mapsto \schurfunctor^\lambda(V)$ corresponding to number partitions $\lambda$, and {\it skew Schur functors} $V \mapsto \schurfunctor^{\lambda/\mu}(V)$ associated to a {\it skew diagram} $\lambda/\mu$; see
 Akin, Buchsbaum, and Weyman \cite{akin1982schur}, Macdonald \cite[Appendix I.A]{Macdonald}.  Certain skew diagrams play an important role here, namely  the {\it ribbon} (= {\it skew/rim hook} = {\it border strip}) diagrams
 that we denote $\sigma(\alpha)$, whose row sizes from bottom to top are specified by an (ordered) {\it composition} $\alpha=(\alpha_1,\alpha_2,\ldots,\alpha_\ell)$, with one column of overlap between consecutive rows.  For example, 
$\sigma((3,1,1,2,4))$ is this diagram:
$$
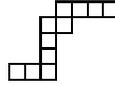
\ytableausetup{boxsize=0.5em}
\begin{ytableau} 
\none& \none& \none & & & &\\ 
 \none& \none&  & \\  
 \none& \none&   \\  
 \none& \none&    \\  
   & &    \\  
\end{ytableau}
$$

\begin{thm}
\label{skew-Schur-functor-theorem}
Fix $d, r \geq 1$ and let $R=\veralg{d}$ and $M=\vermod{d}{ r}$ as above. Then one has an explicit $GL(V)$-equivariant minimal $R$-free resolution of $M$ of the form
$$
0 \leftarrow M 
\leftarrow R \otimes_\kk \schurfunctor^{(r)}(V)
\leftarrow R \otimes_\kk \schurfunctor^{\sigma(d,r)}(V)
\leftarrow R \otimes_\kk \schurfunctor^{\sigma(d,d,r)}(V)
\leftarrow R \otimes_\kk \schurfunctor^{\sigma(d,d,d,r)}(V)
\leftarrow \cdots
$$
whose $i^{th}$ resolvent
$R \otimes_\kk \schurfunctor^{\sigma(d^i,r)}(V)$
has $R$-basis elements in degree $di+r$. 
\end{thm}

\begin{cor}
\label{main-tor-corollary}
In the setting of Theorem~\ref{skew-Schur-functor-theorem}, 
$\Tor_i^R (M,\kk)_j$ vanishes for $j\neq di+r$, and
$$
\Tor_i^R (M,\kk)_{di+r}
\cong \schurfunctor^{\sigma(d^i,r)}(V),
$$
as a polynomial $GL(V)$-representation.
\end{cor}

\begin{example}\label{ex: complex for r=4, d=3}
Taking $d=3$ and $r=4$, 
so that $R=\veralg{3}$ and $M=\vermod{3}{ 4}$, the
$R$-free resolution of $M$ has the form
$$
R \otimes_\kk \schurfunctor^{
\ytableausetup{boxsize=0.3em}
\begin{ytableau}
\, & & & 
\end{ytableau}}(V)
\leftarrow R \otimes_\kk \schurfunctor^{
\ytableausetup{boxsize=0.3em}
\begin{ytableau} 
 \none &\none & & & &\\ 
  & &  
\end{ytableau}
}(V)
\leftarrow R \otimes_\kk \schurfunctor^{
\ytableausetup{boxsize=0.3em}
\begin{ytableau} 
\none& \none& \none &\none & & & &\\ 
 \none& \none&  & & \\  
   & &     
\end{ytableau}
}(V)
\leftarrow \cdots
$$
\end{example}

\begin{example}
We intentionally
allow for the case where $d=1$, so that $R=\veralg{1}=S$, and
$M=\vermod{1}{r}=(x_1,x_2,\ldots,x_n)^r$.
For example, when $d=1$ and $r=4$, 
so that $R=\veralg{1}=S$, then the $S$-free resolution of $M=\vermod{1}{4}=(x_1,x_2,\ldots,x_n)^4$ 
looks like 
$$
S \otimes_\kk \schurfunctor^{
\ytableausetup{boxsize=0.3em}
\begin{ytableau}
\,& & &  
\end{ytableau}}(V)
\leftarrow S \otimes_\kk \schurfunctor^{
\ytableausetup{boxsize=0.3em}
\begin{ytableau} 
\, & & & \\ 
    \\  
\end{ytableau}
}(V)
\leftarrow S \otimes_\kk \schurfunctor^{
\ytableausetup{boxsize=0.3em}
\begin{ytableau} 
\, & & & \\ 
    \\  
    \\     
\end{ytableau}
}(V)
\leftarrow S \otimes_\kk \schurfunctor^{
\ytableausetup{boxsize=0.3em}
\begin{ytableau} 
\, & & & \\ 
    \\  
    \\   
    \\
\end{ytableau}
}(V)
\leftarrow \cdots
$$
This recovers $GL(V)$-equivariant $S$-resolutions of powers $(x_1,\ldots,x_n)^r$ of the irrelevant ideal $S_+$ discussed by  Buchsbaum and Rim \cite{BuchsbaumRim}, Buchsbaum and Eisenbud \cite{buchsbaum1975generic},
and Srinivasan \cite{Srinivasan};  see also Kustin \cite[\S4,5]{kustin2016canonical}.
For $r=1$, it gives a Koszul complex resolving $(x_1,\ldots,x_n)$.
\end{example}

\begin{example}
\label{Veronese-maximal-ideal-example}
Note that since Theorem~\ref{skew-Schur-functor-theorem} 
assumes $r \geq 1$, it deliberately excludes the case $r=0$, where the $R$-free resolution of the module $M=\vermod{d}{0}=\veralg{d}=R$  is trivial.
On the other hand, taking $r=d$,
one has 
$$
M=\vermod{d}{ d}=\veralgplus{d}=R_+,
$$
the irrelevant ideal of $R=\veralg{d}$ whose quotient defines the module $\kk=R/R_+$.  Thus the $R$-free resolution of $M$ also gives the resolution of $\kk$, and 
$\Tor_i^R(\kk,\kk) \cong \Tor_{i-1}^R(R_+,\kk)$.
\end{example}

Building up to the proof of 
 Theorem~\ref{skew-Schur-functor-theorem},
Section~\ref{rep-theory-preliminaries-section}
reviews equivariant resolutions, polynomial $GL(V)$-representations, and skew Schur functors.
Section~\ref{ribbon-schur-functor-section} then
derives some special properties of skew Schur functors of ribbon shape that are crucial for our proof. In particular, given a sequence of ribbons, we construct in Definition \ref{def: Hamel-Goulden-categorification} a right resolution of the Schur functor associated to the \textit{concatenation} of the ribbons in terms of Schur functors associated to various \textit{near-concatenations} of the ribbons. This complex, which may be of independent interest, turns out to be a categorification of a special case of the Hamel-Goulden determinantal identity (see \cite{HamelGoulden}) for skew Schur functions. The modules and differentials of this complex are simple to describe, and analysis of a special case of this complex gives us an essential tool for computing the modules $\Tor_i^R(M,M')$.

The maps in the resolution turn out to be part of an unexpected complex of $S$-modules $(\fancyR_\bullet,\partial)$, dubbed the {\it complex of ribbons} in Section~\ref{ribbon-complex-subsection}.  The differential in the complex of ribbons is the restriction of a simple tensor-degree-lowering map $\partial$ on $S \otimes_\kk T_\kk(S)$, where
here $T_\kk(S)$ is a tensor algebra.  However, this $\partial$ is {\it not} a differential $(\partial^2 \neq 0)$ on all $S \otimes_\kk T_\kk(S)$; rather it only satisfies
$\partial^2=0$ when restricted to the complex of ribbons $\fancyR_\bullet$.
After proving Theorem~\ref{skew-Schur-functor-theorem}
in Section~\ref{resolution-subsection},
the next two subsections discuss a few of its corollaries, including a symmetric function identity (Section~\ref{symmetric-function-identity-section}), and a poset homology calculation
(Section~\ref{poset-homology-section}).

A striking feature of the minimal resolution
in Theorem~\ref{skew-Schur-functor-theorem} is its uniformity, both in the resolvents and the form of its differentials. This uniformity is utilized to great effect in Section~\ref{tor-and-ext-between-modules-section}, where we equivariantly describe 
$\Tor_i^R(M,M')$ and also $\hom_R(M,M')$  
for $R=\veralg{d}$ and all $M,M'$ of the form $\vermod{d}{ r}, \vermod{d}{ r'}$. Even at the level of the standard {\it un-derived} functors $M \otimes_R M'$ and $\hom_R (M,M')$, 
one already sees evidence of the consistent structure 
for these modules in the next two results, proven in Sections~\ref{tensor-section} and \ref{hom-section}.

\begin{thm}\label{thm: tensor-product} 
Fix $d,r,r' \geq 1$ and let $R=\veralg{d}$ with the three $R$-modules 
$$
\begin{aligned}
M&=\vermod{d}{ r},\\
M'&=\vermod{d}{ r'}, \\
M''&=\vermod{d}{ r''}, \quad \text{ where }r''=r+r'.
\end{aligned}
$$
\begin{itemize}
\item[(i)]
The multiplication map $M \otimes_R M' \overset{\varphi}{\rightarrow} M''$ gives rise to a $GL(V)$-equivariant short exact sequence of $R$-modules 
$$
0 \to \bbs^{\sigma(r,r')} (V)(-r'') 
\to 
M \otimes_R M' 
\to M'' \to 0
$$
with the $R$-module 
$\bbs^{\sigma(r,r')} (V)(-r'')$ concentrated in degree $r''$,
annihilated by $R_+$.
\item[(ii)] The sequence splits as $R$-modules, giving an $R$-module
isomorphism
$$
M \otimes_R M' \cong M''
\oplus \schurfunctor^{\sigma(r,r')}(V)(-r'').
$$
\item[(iii)]
When $\binom{r+r'}{r}$ lies in $\kk^\times$,
the sequence also splits as
$\gl (V)$-representations.
\end{itemize}
\end{thm}

\begin{thm}
\label{ext-zero-theorem}
Assume $n \geq 2$, so that $S=\kk[x_1,\ldots,x_n]$
is not univariate\footnote{See Proposition~\ref{prop: one-variable-ext-tor} for the simple answer in the univariate case $n=1$.}.  Fix an integer $d \geq 1$, defining $R=\veralg{d}$.
For $r, r' \geq 0$,
consider three $R$-modules 
$$
\begin{aligned}
M&=\vermod{d}{ r},\\
M'&=\vermod{d}{ r'}, \\
M''&=\vermod{d}{ r''},
\end{aligned}
$$
defining
%in $(r'-r) +d\ZZ$ 
$r'':= r'-r$ if $r \leq r'$,
otherwise if $r>r'$, defining $r''$ to be the unique integer in $[0,d)$ congruent to $r'-r \bmod{d}$.
Then one has a $GL(V)$-equivariant $R$-module isomorphism 
$$
\begin{array}{rcl}
M'' &\longrightarrow &\hom_R (M , M' )\\
m'' &\longmapsto & (m \mapsto m''\cdot m).
%\quad \left( = \Ext^0_R(M,M') \right)
\end{array}
$$
\end{thm}

%\begin{thm}
%Let $d \geq 1$ and $r,r'$ be any integers.
%\begin{enumerate}
%    \item There is a split short exact sequence of $\veralg{d}$-modules
%$$0 \to \bbs^{\sigma(r,r')} (V) \to \vermod{d}{ r} \otimes_{\veralg{d}} \vermod{d}{ r'} \to \vermod{d}{r+r'} \to 0.$$
%\item Assuming that $d$ is a unit in $\kk$, there is an isomorphism
%$$\hom_{\veralg{d}} (\vermod{d}{ r} , %\vermod{d}{ r'} ) \cong \vermod{d}{r''}$$
%for some integer $r''$ (see Theorem . 
%\end{enumerate}
%\end{thm}

Section~\ref{tor-between-the-modules-subsection}  uses Theorem~\ref{skew-Schur-functor-theorem} to equivariantly describe $\Tor_i^R(M,M')$ for $i \geq 1$, which is even simpler.  Note that since $\kk = R/R_+$ where $R_+=\vermod{ d}{d}$,  dimension-shifting
lets one rephrase Theorem~\ref{skew-Schur-functor-theorem} as
saying that for $M,M'$ being $\vermod{d}{r}, \vermod{d}{d}$, one has 
$
\Tor_{i-1}^R(M,M') \cong \schurfunctor^{\sigma(d,d^{i-1},r)}(V)
$
for $i \geq 1$.
Our last main result elegantly generalizes this.

\begin{thm}
\label{higher-tor-is-all-socle-and-Artinian}
Fix $d,r,r' \geq 1$,
and let $R,M,M'$ denote
$\veralg{d}, \vermod{d}{ r}, \vermod{d}{ r'}$, as usual.
Then for $i \geq 1$, the $R$-module $\Tor^R_i(M,M')$ is annihilated by $R_+$, and as a module over $\kk=R/R_+$, has a $GL(V)$-isomorphism
$$
\Tor^R_i(M,M')
\cong
\schurfunctor^{\sigma(r,d^i,r')}(V).
$$
\end{thm}

\begin{example}
Note that Theorem~\ref{higher-tor-is-all-socle-and-Artinian} is interesting even if $d=1$, so 
$R = S= \kk[x_1 , \dots , x_n]$. It asserts, for instance, that there is a $\gl (V)$-equivariant isomorphism
$$
\tor_3^S ((S_+)^2 , (S_+)^3) \cong \bbs^{{\tiny \ydiagram{1+3,1+1,1+1,1+1,2}}} (V).
$$
\end{example}

%%%%%%%%%%%%%%%%%%%%%%%%%%%%
\section*{Acknowledgements}
The authors thank Francesca Gandini for helpful conversations, and Darij Grinberg for helpful edits, 
including a shortening of the proof of Proposition~\ref{skew-concatenation-near-concatenation-prop}(iii). They thank an anonymous referee for helpful edits. The first author was partially supported by NSF grant DMS-1745638, third and fourth authors by NSF grant DMS-2053288, and the fifth author by NSF grant DMS-2202871.
%%%%%%%%%%%%%%%%%%%%%%%%%%%%

%%%%%%%%%%%%%%%%%%%%%%%%%%%%
\section{Representation theory preliminaries}
\label{rep-theory-preliminaries-section}
%%%%%%%%%%%%%%%%%%%%%%%%%%%%

%%%%%%%

%%%%%%%
\subsection{Some generalities on equivariant resolutions}
\label{equivariant-resolutions-section}

We mention a fact about equivariant resolutions that can be found in Broer, Reiner, Smith and Webb \cite[\S2]{BRSW}.
%To state them requires a bit of set-up.  Let $U$ be a finite-dimensional $\kk$-vector space carrying a $\kk$-linear representation of a $G$;  by a standard abuse of notation, we will refer to the representation itself as $U$.  Then $U$ also gives rise to an element $[U]$ in the {\it Grothendieck ring} $G_0(\kk G)$ of finite-dimensional $\kk G$-modules.  The additive structure of $G_0(\kk G)$ is defined as the quotient of the free abelian group having as basis these elements $[U]$ running through all isomorphism classes of $\kk G$-modules, then imposing relations $[U_2]=[U_1]+[U_3]$ for all short exact sequences of $\kk G$-modules $0 \rightarrow U_1 \rightarrow U_2 \rightarrow U_3 \rightarrow 0$.  The ring structure is induced from the definition $[U][U']:=[U \otimes_\kk U']$.
%For any graded $\kk$-vector space $U=\bigoplus_{j \geq j_0} U_j$ with $U_{j_0} \neq 0$, define $\start(U):=j_0$.   If $\dim_\kk U_j$ is a finite-dimensional $\kk G$-module for all $j$, then $U$ gives rise to an element $$[U](t):=\sum_{j} [U_j]t^j$$  in the formal powers series ring $G_0(\kk G)[[t]]$.
Let $R$ be a commutative, graded Noetherian $\kk$-algebra, and $G$ a group of $\kk$-algebra automorphisms of $R$.  For
$M$ a Noetherian graded $R$-module with a $G$-action via
$\kk$-module automorphisms, say 
$G$ acts {\it compatibly}\footnote{Alternatively, $M$ is a module for the {\it skew group algebra} $R \#G$; see Leuschke and Wiegand \cite{LeuschkeWiegand}. } on $M$ if $g(r \cdot m)=g(r) g(m)$ for
all $r$ in $R$, $m$ in $M$, $g$ in $G$.

\begin{prop}
%\cite[\S 2.1, 2.2, 2.3]{BRSW}
\label{BRSW-prop}
In the above setting, one has the following.
%\begin{itemize}
If $M,N$ are two graded Noetherian $R$-modules with compatible $G$-actions as above, then $\Tor_i^R(M,N)$ and $\Ext^i_R(M,N)$ for all $i \geq 0$ are also
    graded Noetherian $R$-modules with compatible $G$-actions.
%    \item[(ii)] One has $\start(\Tor_i(M,\kk)) < \start(\Tor_{i+1}(M,\kk))$ for all $i \geq 0$.    Consequently, in $G_0(\kk G)[[t]]$, this element is well-defined,
%$$
%    \sum_{i \geq 0} (-1)^i [\Tor_i^R(M,\kk)](t).
%    $$
%    \item[(iii)] Furthermore, one has this idenity in the ring $G_0(\kk G)[[t]]$:
%    $$
%    \sum_{i \geq 0} (-1)^i [\Tor_i^R(M,\kk)](t)= \frac{[M](t)}{[R](t)}.
%    $$
%\end{itemize}
\end{prop}

%%%%%%%
\subsection{Review of polynomial $GL(V)$ representations}
\label{general-linear-review-section}

We will need some of the basic facts on 
polynomial representations of $GL_n(\kk)$,
and polynomial functors.  References are
Akin, Buchsbaum and Weyman \cite{akin1982schur}, Fulton \cite[\S8.2,8.3]{Fulton}, Fulton and Harris \cite[Lec. 6]{FultonHarris}, Gandini \cite{gandini2018resolutions}, Macdonald \cite[Appendix I.A]{Macdonald}, Stanley \cite[App. 7.A.2]{Stanley-EC2}.

Let $V=\kk^n$ be a free $\kk$-module of rank $n$, so that $GL(V) \cong GL_n(\kk)$.
A representation $\varphi: GL(V) \rightarrow GL(U)$ (where $U$ is also a free $\kk$-module)
is {\it polynomial} if the matrix entries $\left[ \varphi(A)_{k,\ell}\right]_{k,\ell=1,2,\ldots}$ are all
polynomial functions in the matrix entries of $A=[a_{ij}]_{i,j=1,2,\ldots,n}$.
If these entries $\varphi(A)_{k,\ell}$ are all homogeneous in $\{a_{ij}\}$ of degree $d$, then one says that $\varphi$ is a {\it homogeneous} polynomial representation of degree $d$.  When no confusion arises, we will again refer to $\varphi$ by $U$, in our standard abuse of notation.  A polynomial representation $U$ has
{\it character}
$$
\ch_U=\ch_U(x_1,\ldots,x_n)=
\trace \varphi(\diag(\xx))
%\trace(\diag(\xx)|_U)
$$
where 
$\diag(\xx)$ is an $n \times n$ diagonal matrix having
eigenvalues $x_1,\ldots,x_n$.
%$$\diag(\xx)=\left[ \begin{smallmatrix}x_1 & & & \\ &x_2& & \\&   &\ddots& \\&   &      &x_n\end{smallmatrix}\right]$.  
This character $\ch_U$ is a symmetric polynomial in $x_1,x_2,\ldots,x_n$.  However, all of the representations of $GL(V)$ that we consider here are restrictions to $GL(V)$ of {\it polynomial functors} from the category of finite-rank free $\kk$-modules and $\kk$-linear maps to itself.  This means that they are defined for all $n$, and one can regard $ch_U$ as an element of the {\it ring of  symmetric functions} $\Lambda_\kk$ in the
infinite variable set $x_1,x_2,\ldots$, with coefficients in $\kk$.  When $\kk$ is a field of characteristic zero, the $GL(V)$-representation $U$ is determined up to isomorphism by
its character $\ch_U$. 

\begin{remark}
For $\kk$ a field, the polynomial functors that we are considering are {\it strict polynomial functors} in the sense of Friedlander and Suslin \cite[Defn.~2.1]{FriedlanderSuslin}.  Although we will not make use of it, a more general theory of strict polynomial functors over arbitrary commutative rings $\kk$ is described by Krause \cite[Chap.~8]{Krause}.
\end{remark}

Here are some examples of polynomial functors and their characters.
\begin{example}
The $m^{th}$ {\it tensor power} $T^m(V)=V^{\otimes m}$, the $m^{th}$ {\it symmetric power} $S^m(V)$ and the $m^{th}$ {\it exterior power} $\wedge^m(V)$ are all
polynomial functors, with 
$$
\begin{aligned}
\ch_{T^m(V)}&=h_1^m =e_1^m=(x_1+x_2+\cdots)^m,\\
\ch_{S^m(V)}&=h_m,\\
\ch_{\wedge^m(V)}&=e_m.
\end{aligned}
$$
Here $h_m$ is the {\it complete homogeneous} symmetric function, which is the
sum of all monomials of degree $m$ in $x_1,x_2,\ldots$,
and
$e_m$ is the {\it elementary} symmetric function, which is the
sum of all {\it squarefree} monomials of degree $m$ in $x_1,x_2,\ldots$.
\end{example}

\begin{example}\label{schur-functor-example}
For each partition $\lambda$ with $m=|\lambda|:=\sum_i \lambda_i$,
there is a degree $m$ homogeneous polynomial representation of $GL(V)$ called the {\it Schur functor} $\schurfunctor^\lambda(V)$,
with
$$
\ch_{\schurfunctor^\lambda(V)}=s_\lambda,
$$
where $s_\lambda$ is
the {\it Schur function} in 
$x_1,x_2,\ldots$.
When $\kk$ is a field of characteristic zero,  Schur functors $\schurfunctor^\lambda(V)$
give all of the {\it irreducible}
polynomial $GL(V)$-representations.
\end{example}

\begin{example}
\label{schur-functor-example2}
The construction of $\schurfunctor^\lambda(V)$ is a special
case of a {\it Schur functor} $\schurfunctor^{D}(V)$
defined for more general diagrams $D \subset \{1,2,\ldots\} \times \{1,2,\ldots\}$, such as when $D=\lambda/\mu$ is a {\it skew Ferrers diagram} for a pair of partitions $\mu \subseteq \lambda$.  The skew Schur functor $\schurfunctor^{\lambda/\mu}(V)$ has 
$\ch_{\schurfunctor^{\lambda/\mu}(V)}=s_{\lambda/\mu}$, the {\it skew Schur function}. 
We illustrate the construction of the Schur functor $\schurfunctor^D(V)$ here with an example.

Label the cells of $D$ bijectively with $1,2,\ldots,m$, and define the sets of labels $C_1,\ldots,C_c$ in its columns (left-to-right)
and rows $R_1,\ldots,R_r$ (bottom-to-top), as in this example:
$$
D=
\ytableausetup{boxsize=1.0em}
\begin{ytableau}
 \none&\none&\none&7&8 \\  
 \none&3&5&6\\  
 1&2&4 
\end{ytableau}
\qquad
\begin{aligned}
& \\
&R_1=\{1,2,4\}, R_2=\{3,5,6\}, R_3=\{7,8\},\\
&C_1=\{1\},C_2=\{2,3\},C_3=\{4,5\},C_4=\{6,7\},C_5=\{8\}.
\end{aligned}
$$
Introduce the {\it row} and {\it column compositions} for $D$:
$$
\begin{aligned}
\rows(D)&:=(|R_1|,\ldots,|R_r|),\\
\cols(D)&:=(|C_1|,\ldots,|C_c|).
\end{aligned}
$$
Given a composition $\alpha=(\alpha_1,\ldots,\alpha_\ell)$, abbreviate 
$$
\begin{aligned}
S^{\alpha}(V)&:=S^{\alpha_1}(V) \otimes_\kk \cdots \otimes_\kk S^{\alpha_\ell}(V),\\
\wedge^{\alpha}(V)&:=\wedge^{\alpha_1}(V) \otimes_\kk \cdots \otimes_\kk \wedge^{\alpha_\ell}(V).
\end{aligned}
$$
Then $\schurfunctor^D(V)$ will be a subfunctor
of $S^{\rows(D)}(V)$
%:=S^{|R_1|}(V) \otimes_\kk \cdots \otimes_\kk S^{|R_r|}(V)$ 
defined as the
image of the composite of two maps,
the first an {\it antisymmetrization} map $\fancyA_D$ along the columns of $D$, the second a {\it symmetrization} map $\fancyB_D$ along the rows of $D$:
$$
\wedge^{\cols(D)}(V) 
\overset{\fancyA_D}{\longrightarrow}
T^m(V)
\overset{\fancyB_D}{\longrightarrow}
S^{\rows(D)}(V)
$$
Here the maps $\fancyA_D,\fancyB_D$ are described as follows.  Assuming that the column indexing sets $C_1,\ldots,C_c$ each
consist of contiguous sequences of integers, let $w_D$ be any permutation of the
tensor positions that sends the row indexing sets
$R_1,\cdots,R_c$ to a contiguous sequence of integers. Then $\fancyA_D,\fancyB_D$ can be described via tensor products as
$$
\begin{aligned}
\fancyA_D&:=\bigotimes_{j=1}^c \fancyA_{|C_j|},\\
\fancyB_D&:= \left( \bigotimes_{i=1}^r \fancyB_{|R_i|} \right)
\circ w_D,
\end{aligned}
$$
where $\fancyA_\ell: \wedge^\ell(V) \rightarrow T^\ell(V)$ is the usual inclusion via antisymmetrization 
sending
$$
v_1 \wedge v_2 \wedge \cdots \wedge v_\ell \longmapsto
\sum_{w \in \symm_\ell} \sgn(w) \cdot v_{w(1)} \otimes v_{w(2)} \otimes \cdots v_{w(\ell)},
$$
and $\fancyB_\ell: T^\ell(V) \rightarrow S^{\ell}(V)$ is the usual symmetrization/multiplication map
sending 
$$
v_1 \otimes v_2 \otimes \cdots \otimes v_\ell 
\longmapsto v_1 \cdot v_2 \cdots  v_\ell.
$$
In our example above, the map $\fancyA_D$ is
$$
\wedge^{\cols(D)}(V)=\wedge^1 (V) \otimes
\wedge^2(V) \otimes
\wedge^2(V) \otimes
\wedge^2(V) \otimes
\wedge^1(V) 
\rightarrow T^8(V)
$$
sending
$$
\begin{aligned}
&v_1 \otimes
(v_2 \wedge v_3) \otimes
(v_4 \wedge v_5) \otimes
(v_6 \wedge v_7) \otimes
v_8 \\
&\longmapsto
v_1 \otimes
(v_2 \otimes v_3-v_3 \otimes v_2) \otimes
(v_4 \otimes v_5-v_5 \otimes v_4) \otimes
(v_6 \otimes v_7-v_7 \otimes v_6) \otimes
v_8
\end{aligned}
$$
while the map $\fancyB_D$ is this permutation $w_D$ and tensor product of symmetrizations:
$$
\begin{array}{rcccccl}
V^{\otimes 8}
&\longrightarrow &
S^3(V) &\otimes&
S^3(V) &\otimes&
S^2(V)  =S^{\rows(D)}(V)\\
u_1 \otimes  \cdots
\otimes u_8
&\longmapsto &
u_1 u_2 u_4 &\otimes& u_3 u_5 u_6 &\otimes& u_7 u_8.
\end{array}
$$
As notation, given a filling $T$ of the squares of $D$ with vectors from $V$, one can associate to $T$ a
tensor of decomposable wedges in
$\wedge^{\cols(D)}(V)$, 
and then let $[T]$
denote its image in $S^{\rows(D)}(V)$
%S^{|R_1|}(V) \otimes \cdots \otimes S^{|R_r|}(V)$ 
under $\fancyB_D \circ \fancyA_D$.
For example, with $D$ the diagram above,
$$
\left[
\ytableausetup{boxsize=1.2em, centertableaux}
\begin{ytableau}
 \none&\none&\none&v_7&v_8 \\  
 \none&v_3&v_5&v_6\\  
 v_1&v_2&v_4 
\end{ytableau}
\right]
= \fancyB_D\left(\fancyA_D\left(\,\,
v_1 \otimes
(v_2 \wedge v_3) \otimes
(v_4 \wedge v_5) \otimes
(v_6 \wedge v_7) \otimes
v_8 \,\,
\right)\right).
$$
\end{example}

Akin, Buchsbaum, and Weyman prove several
important properties of the Schur functors $\schurfunctor^{\lambda/\mu}$
corresponding to skew shapes $\lambda/\mu$ in \cite[Theorem II.2.16]{akin1982schur}.  First, they show $\schurfunctor^{\lambda/\mu}(V)$ has an explicit
    presentation as a quotient representation of $\wedge^{\cols(\lambda/\mu)}(V)$,
rather than as a subrepresentation of 
$S^{\rows(\lambda/\mu)}(V)$.
Specifically, they show that the kernel of the map
$$ 
\fancyB_{\lambda/\mu} \circ \fancyA_{\lambda/\mu}:
\wedge^{\cols(\lambda/\mu)}(V) \longrightarrow
S^{\rows(\lambda/\mu)}(V)
$$
is generated by certain {\it Garnir relations}.
Secondly, they show that $\schurfunctor^{\lambda/\mu}(V)$
is not only a free $\kk$-module, but \textit{universally free}, in the sense that it has a $\kk$-basis that commutes with any change of the base ring $\kk$,
described as follows.  Once one picks a $\kk$-basis $e_1,\ldots,e_n$ for $V$, one obtains an obvious 
monomial $\kk$-basis for $\wedge^{\cols(D)}(V)$ consisting of tensors of wedges of the $\{e_i\}$, indexed by fillings $T: \lambda/\mu \rightarrow \{1,2,\ldots,n\}$  that are {\it column-increasing}, that is, strictly increasing from top-to-bottom in each column of $\lambda/\mu$. Letting $[T]$ denote the image of
such an element under $\fancyB_{\lambda/\mu}\circ \fancyA_{\lambda/\mu}$, they show $\schurfunctor^{\lambda/\mu}(V)$ has a $\kk$-basis 
$\{ [T] \}$ as $T$ runs through the subset of {\it semistandard (column-strict) tableaux} of shape $\lambda/\mu$:  fillings $T$ which are strictly increasing top-to-bottom down each column, but also weakly increasing left-to-right in each row.   Garnir relations let one rewrite any $[T]$
where $T$ is a column-increasing filling as a $\kk$-linear combination of $[T']$
indexed by semistandard tableaux $T'$.

We mention here a consequence of this universal freeness which we will need.

\begin{prop}
\label{tensoring-with-skews-is-exact-prop}
Fixing a skew shape $\lambda/\mu$, for any complex $\mathcal{C}$ of $GL(V)$-modules, tensoring
over $\kk$ with $\schurfunctor^{\lambda/\mu}(V)$ is exact, so it commutes with taking homology: one has an 
isomorphism of $GL(V)$-modules
$$
H_i(\mathcal{C} \otimes_\kk \schurfunctor^{\lambda/\mu}(V))
\cong 
H_i(\mathcal{C}) \otimes_\kk \schurfunctor^{\lambda/\mu}(V).
$$
\end{prop}

\begin{example}
When $\kk$ is a field,
tensor products, subfunctors and quotient functors of polynomial functors remain polynomial.  Hence the homology groups of a complex of polynomial functors are all polynomial functors. 

In particular, in our setting where $R=\veralg{d}$ and $M=\vermod{d}{ r}$ have each of their homogeneous components a polynomial functor of the form $S^j(V)$, one can conclude that homogeneous components of $\Tor_i^R(M,\kk)=\Tor_i(\kk,M)$ are polynomial functors, as they can be computed as the homology of a tensored Koszul complex $(R \otimes_\kk \wedge V) \otimes_R M$. Similar arguments show that the ($\NN$-graded) homogeneous components $\Tor_i^R(M,M')_j$ of each $\Tor_i^R(M,M')$ are polynomial functors, where $M=\vermod{d}{ r}$ and $M'=\vermod{d}{r'}$. This will also follow from the polynomial form of the $R$-free resolution of $M$ in  Theorem~\ref{skew-Schur-functor-theorem}.

Note that in this setting, $\hom_R (M,N)$ (and hence $\ext_R^i (M , \kk)$) is not necessarily a polynomial functor. 
The action of an element $g$ in $GL(V)$ on $\varphi$ in $\hom_R (M , N)$
via $\varphi \mapsto g \circ \varphi \circ g^{-1}$ introduces denominators of $\det(g)$ into the action.
\end{example}

%Polynomial representations $U$ of $GL(V)$ are {\it completely reducible}: 
%$$ U = \bigoplus_{\lambda} \schurfunctor^{\lambda}(V)^{\oplus m^U_\lambda}.$$
%The irreducible multiplicities $m^U_\lambda$ here are determined by the unique expansion
%$$\ch_U = \sum_\lambda m^U_\lambda s_\lambda.$$
%Applying this in the context of Proposition~\ref{BRSW-prop}(iii), one concludes that if the $j^{th}$ homogeneous components of the graded ring $R$ and module $M$ are homogeneous polynomial $GL(V)$-representations of degree $j$, then one has this character identity within the ring of symmetric functions $\Lambda$:
%\begin{equation}
%    \label{Tor-is-quotient-of-Hilbs}
%\sum_{i=0}^\infty (-1)^i \ch( \Tor_i^R(M,\kk) ) 
%=\frac{\ch(M)}{\ch(R)}.
%\end{equation}

%%%%%%
\subsection{Schur-Weyl duality}
\label{Schur-Weyl-duality-section}
Every homogeneous polynomial functor $V \mapsto \PP(V)$ of degree $m$ 
not only gives rise to
a family of polynomial representations of 
$GL(V)=GL_n(\kk)$ for all $n$, but 
also a representation of $\symm_m$, or $\kk\symm_m$-module,
on its {\it squarefree} or {\it multilinear} or $x_1 x_2 \cdots x_m$-{\it weight
space}:  pick $V=\kk^m$, and consider the $\kk$-vector space
$$
\PP(V)_{x_1x_2 \cdots x_m}
:=\{ u \in \PP(V): \diag(\xx) u = x_1 x_2 \cdots x_m \cdot u\}.
$$
Conversely, when $\kk$ has characteristic zero, one can recover the functor $V \mapsto \PP(V)$ from this 
$\kk \symm_m$-module via the formula
$$
\PP(V)
\cong \bigoplus_{m=0}^\infty  \left( T^m(V)  \otimes_{\kk \symm_m} \PP(V)_{x_1 \cdots x_m} \right)
$$
where the $m$-fold tensor product
$T^m(V) =V^{\otimes m}$
%= \underbrace{V \otimes \cdots \otimes V}_{m\text{factors}},
is regarded as a left $\kk GL(V)$-module under the diagonal action
$$
g(v_1 \otimes \cdots \otimes v_m):=g(v_1) \otimes \cdots \otimes g(v_m)
$$
and a right $\kk{\symm_n}$-module using the tensor position action
$$
(v_1 \otimes \cdots \otimes v_m)\sigma:=v_{\sigma(1)} \otimes \cdots \otimes v_{\sigma(m)}.
$$
The $GL(V)$-irreducible Schur functor $\schurfunctor^\lambda(V)$ with $|\lambda|=m$
has $x_1 \cdots x_m$-weight space $\schurfunctor^\lambda(V)_{x_1 \cdots x_m}$ isomorphic to the irreducible Specht module $\specht^\lambda$ for $\symm_m$.
More generally, each skew shape $\lambda/\mu$ with $m=|\lambda|-|\mu|$ has skew Schur functor $\schurfunctor^{\lambda/\mu}(V)$ with
$x_1 \cdots x_m$-weight space $\schurfunctor^{\lambda/\mu}(V)_{x_1 \cdots x_m}$ isomorphic to the {\it skew Specht module} $\specht^{\lambda/\mu}$ for $\symm_m$.

\vskip.1in
\noindent
{\bf Notational warning.}
To lighten notation, from here onward we often omit the $V$
from the Schur functor notation $\schurfunctor^{\lambda/\mu}(V)$,
and denote this simply by $\schurfunctor^{\lambda/\mu}$.  
In particular, 
$$
\schurfunctor^{(m)}=\schurfunctor^{(m)}(V)=S^m(V)=S^m=S_m
$$
where $S=\kk[x_1,\ldots,x_n]=S(V)$.
We hope context will resolve any confusion.

%%%%%%%%%%%%%%%%%%%%%%%%%%%%%%%%%%%%%%%
\section{Ribbon Schur Functors}
\label{ribbon-schur-functor-section}
%%%%%%%%%%%%%%%%%%%%%%%%%%%%%%%%%%%%%%%

Schur functors coming from \textit{ribbon} diagrams
play a central role in this work.  These Schur functors are almost never irreducible as $GL(V)$-representations, but still give some of the simplest classes of Schur functors one can study.
Our goal in the next few subsections is to present characteristic-free
homological results
valid for arbitrary commutative rings $\kk$, lifting
symmetric function
identities known for ribbon
Schur functions.  We will make heavy use of the universal freeness results for skew Schur functors of Akin, Buchsbaum and Weyman \cite{akin1982schur}, e.g., Proposition~\ref{tensoring-with-skews-is-exact-prop}.

%%%%
\subsection{Concatenation
and near-concatenation}
%%%%

\begin{definition} \rm
Given a composition $\alpha = (\alpha_1,\dots, \alpha_\ell)$, recall that 
the ribbon diagram $\sigma(\alpha)$ has row sizes from bottom to top given by the entries of $\alpha$, and each consecutive row overlaps in exactly one column.
For two compositions $\alpha = (\alpha_1,\dots, \alpha_\ell), \beta = (\beta_1,\dots, \beta_m)$, define their \textit{concatenation} $\alpha\cdot \beta$ and  \textit{near concatenation}
$\alpha\odot\beta$ by
$$
\begin{aligned}
\alpha\cdot \beta &:= (\alpha_1,\dots, \alpha_\ell, \beta_1,\dots, \beta_m),\\
\alpha\odot\beta &:= (\alpha_1,\dots, \alpha_{\ell-1}, \alpha_\ell+\beta_1, \beta_2,\dots, \beta_m).
\end{aligned}
$$
\end{definition}

\begin{example}
If
$\alpha=(2,1,3)$ and $\beta=(2,4,2)$, then one has
$\alpha\cdot \beta =(2,1,3,2,4,2)$ and
$\alpha\odot \beta =(2,1,5,4,2)$ with ribbon diagrams
$$
\sigma(\alpha)=
\ytableausetup{boxsize=0.8em}
\begin{ytableau} 
 \none& \alpha&\alpha&\alpha   \\  
 \none& \alpha    \\  
\alpha&\alpha     
\end{ytableau}
\qquad
\sigma(\beta)=
\ytableausetup{boxsize=0.8em}
\begin{ytableau} 
\none& \none& \none& \none&\beta&\beta\\ 
 \none& \beta&\beta&\beta&\beta\\  
 \beta& \beta
\end{ytableau}
$$
$$
\sigma(\alpha \cdot \beta)=
\ytableausetup{boxsize=0.8em}
\begin{ytableau} 
\none&\none&\none&\none& \none& \none& \none&\beta&\beta\\ 
\none&  \none& \none&\none& \beta&\beta&\beta&\beta\\  
\none&  \none&\none & \beta& \beta\\
\none& \alpha&\alpha&\alpha   \\  
\none& \alpha\\  
\alpha&\alpha    
\end{ytableau}
\qquad
\sigma(\alpha \odot \beta)=
\ytableausetup{boxsize=0.8em}
\begin{ytableau} 
\none&\none&\none&\none&\none& \none& \none& \none&\beta&\beta\\ 
\none&\none&  \none& \none&\none& \beta&\beta&\beta&\beta\\
  \none& \alpha&\alpha&\alpha&\beta&\beta   \\  
 \none& \alpha    \\  
\alpha&\alpha    
\end{ytableau}
$$
\end{example}

There is a family of basic relations among ribbon Schur functions \cite[Chapter I.5, Example 21(a)]{Macdonald},
known to generate all other relations among them \cite[Prop. 2.2]{BilleraThomasVanWilligenburg}:
\begin{equation}
\label{concatenation-near-concatenation-identity}
s_{\alpha} s_{\beta} = s_{\alpha \cdot \beta} + s_{\alpha \odot \beta}
\end{equation}
We wish to lift this identity
\eqref{concatenation-near-concatenation-identity} to the following short exact sequence, which will form the base case in
an inductive proof of 
the more general Theorem~\ref{lem:HGcatComplex} below.

\begin{prop}
\label{concatenation-near-concatenation-prop}
For any compositions $\alpha, \beta$ one has maps
of polynomial functors $\Delta_{\alpha,\beta}, m_{\alpha,\beta}$ giving rise to 
a $GL(V)$-equivariant short exact sequence of representations:
\begin{equation}
\label{eq: concatenation SES}
    0 \ra \schurfunctor^{\sigma(\alpha\cdot \beta)} \xrightarrow[]{\Delta_{\alpha,\beta}} \schurfunctor^{\sigma(\alpha)}\otimes \schurfunctor^{\sigma(\beta)} \xrightarrow[]{m_{\alpha,\beta}} \schurfunctor^{\sigma(\alpha\odot\beta)}\ra 0.
\end{equation}
This sequence is split over $\kk$, but not necessarily split over $GL (V)$. 
\end{prop}

When the compositions $\alpha$ and $\beta$ as in the statement of Proposition \ref{concatenation-near-concatenation-prop} are clear from context, we will use the more concise notation $m, \Delta$ to denote the maps $m_{\alpha, \beta} , \Delta_{\alpha, \beta}$.

In fact, it is not hard to generalize concatenation, near-concatenation, the identity \eqref{concatenation-near-concatenation-identity} and sequence \eqref{eq: concatenation SES} to more general classes of diagrams.

\begin{definition} 
\label{skew-concat-defn}
\rm 
Given two (not necessarily skew-shaped) diagrams $D, D'$, define their {\it disjoint sum} $D \oplus D'$, {\it concatenation}
$D \cdot D'$ and {\it near-concatenation} $D \odot D'$ as follows, illustrated for two skew diagrams here, with two extreme cells labeled $x,x'$ for clarity:
$$
D=
\ytableausetup{boxsize=0.8em}
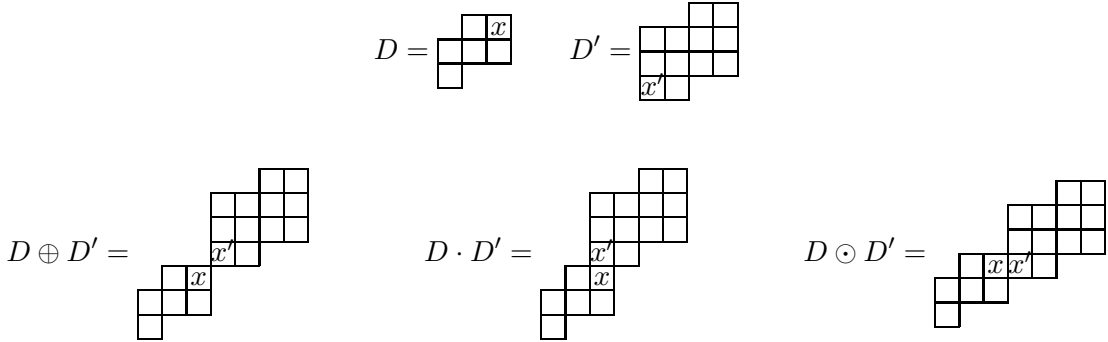
\begin{ytableau} 
 \none& &x\\  
      & &     \\  
\\     
\end{ytableau}
\qquad
D'=
\ytableausetup{boxsize=0.8em}
\begin{ytableau} 
\none&\none& & \\
     &     & & \\ 
     &     & & \\ 
    x'& \\
\end{ytableau}
$$
\vskip.2in
$$
D \oplus D'=
\ytableausetup{boxsize=0.8em}
\begin{ytableau} 
\none&\none&\none&\none&\none& & \\
\none&\none&\none&     &     & & \\ 
\none&\none&\none&     &     & & \\ 
\none&\none&\none&x'    & \\
\none& &x\\  
     & &     \\  
\\  
\end{ytableau}
\qquad \qquad
D \cdot D'=
\ytableausetup{boxsize=0.8em}
\begin{ytableau} 
\none&\none&\none&\none& & \\
\none&\none&     &     & & \\ 
\none&\none&     &     & & \\ 
\none&\none&x'    & \\
\none& &x\\  
     & &     \\  
\\ 
\end{ytableau}
\qquad \qquad
D \odot D'=
\ytableausetup{boxsize=0.8em}
\begin{ytableau} 
\none&\none&\none&\none&\none& & \\
\none&\none&\none&     &     & & \\ 
\none&\none&\none&     &     & & \\ 
\none& &x&x'    & \\
     & &     \\  
\\ 
\end{ytableau}
$$
\begin{itemize}
    \item 
The disjoint sum $D \oplus D'$ places $D$ southwest of $D'$, sharing no rows and columns. 
\item The concatenation $D \cdot D'$ is obtained from
$D \oplus D'$ by moving the cells of the rightmost column of $D$ and leftmost column of $D'$ to lie in the same column,  with no other overlap of rows and columns.
\item The near-concatenation $D \odot D'$ is obtained from
$D \oplus D'$
by moving the cells in the top row of
$D$ and bottom row of $D'$ to
lie in the same row,  with no other overlap of rows and columns.
\end{itemize}
\end{definition}

The following is an easy generalization of identity \eqref{concatenation-near-concatenation-identity}.
\begin{prop}
\label{skew-concatenation-near-concatenation-identity}
For any two skew shapes $D, D'$, one has the skew Schur function identity
$$
s_{D} s_{D'} = s_{D \cdot D'}
+ s_{D \odot D'}.
$$
\end{prop}
\begin{proof}
Use the formula the skew Schur function as
$
s_D=\sum_T \xx_T
$
where $T$ runs over all column-strict tableaux of shape $D$ with entries in $\{1,2,\ldots,n\}$.  Here 
$
\xx_T:=\prod_{i=1}^n x_i^{c_i(T)}
$
where $c_i(T)$ is the number of occurrences of $i$ in the tableau $T$.

The left side in the proposition can be written as the sum $\sum_{(T,T')} \xx_T \xx_{T'}$
over all pairs $(T,T')$ of column-strict tableaux of shape $D, D'$.  Letting $i,i'$, respectively, denote the entries filling the northeastmost cell $x$ in $T$, and the
southwestmost cell $x'$ in $T'$, the two summands on the right side of the identity segregate these $(T,T')$ summands on the left side
according to whether
\begin{itemize}
\item 
$i > i'$, so the tableaux $T, T'$ glue to give one of shape $D \cdot D'$, or 
\item 
$i \leq i'$, so the tableaux $T, T'$ glue to give one of shape $D \odot D'$. \qedhere
\end{itemize}
\end{proof}

The proof of Proposition~\ref{skew-concatenation-near-concatenation-identity} does not extend to characters $\ch_{\bbs^D(V)}$ associated with {\it arbitrary} (not necessarily skew)  diagrams $D$, since in general the tableaux with semistandard fillings need not form a basis for the Schur functor $\bbs^D (V)$ (see for instance \cite[Remark (ii) on page 489]{woodcock1994vanishing}). We next lift the identity in the Proposition~\ref{skew-concatenation-near-concatenation-identity} to an exact sequence.

\begin{prop}
\label{skew-concatenation-near-concatenation-prop}
Let $D, D'$ be any diagrams, not necessarily skew-shaped.
\begin{enumerate}
    \item[(i)] There is an injective $GL(V)$-equivariant map $\Delta: \schurfunctor^{D \cdot D'} \hookrightarrow \schurfunctor^{D \oplus D'}$.
    \item[(ii)] There is a surjective $GL(V)$-equivariant map $m: \schurfunctor^{D \oplus D'} \twoheadrightarrow \schurfunctor^{D \odot D'}$.
    \item[(iii)] Assume $D$ contains a northeastmost corner cell $x$
    and $D'$ contains a southwestmost corner cell $x'$, such as when $D,D'$ are skew shapes.  Then the maps $m, \Delta$ in (i),(ii) satisfy $m \circ \Delta=0$, giving a short complex,
    but not always exact at the middle term:
    $$
    0 \longrightarrow 
    \schurfunctor^{D \cdot D'} 
    \overset{\Delta}{\longrightarrow}
    \schurfunctor^{D \oplus D'}
    \overset{m}{\longrightarrow}
    \schurfunctor^{D \odot D'}
    \longrightarrow 0.
    $$
    \item[(iv)] When $D,D'$ are both skew shapes,
    the complex in (iii) is short exact and split as $\kk$-modules (but not necessarily split as $GL (V)$-modules).  
\end{enumerate}
\end{prop}

\noindent
In particular, part (iv) applied to $D=\sigma(\alpha), D'=\sigma(\beta)$ gives
Proposition~\ref{concatenation-near-concatenation-prop}, and the maps $\Delta, m$ are the maps $\Delta_{\alpha,\beta}, \,m_{\alpha,\beta}$.

\begin{example}
For diagrams
$
D=\ytableausetup{boxsize=0.8em}
\begin{ytableau} 
x \\ 
\end{ytableau}
$
and 
$ 
D'=
\begin{ytableau}
\, \\
x'& \\
\end{ytableau}
$,
the complex 
%$0 \rightarrow \bbs^{D \cdot D'} \rightarrow \bbs^{D \oplus D'} \rightarrow \bbs^{D \odot D'} \rightarrow 0$
in (iii) takes the form
\begin{equation}
\label{non-exact-short-complex}
0 \rightarrow 
\bbs^{
\ytableausetup{boxsize=0.4em}
\begin{ytableau} 
\, \\
 &  \\
\,
\end{ytableau}
}
\overset{\Delta}{\longrightarrow}
\bbs^{\ytableausetup{boxsize=0.4em}
\begin{ytableau} 
\none& \\
\none & & \\
\,\\
\end{ytableau}}
\overset{m}{\longrightarrow}
\bbs^{\ytableausetup{boxsize=0.4em}
\begin{ytableau} 
\none& \\
 & & \\
\end{ytableau}}
\rightarrow
0.
\end{equation}
Here one can check that the map $\Delta$ sends
$$
\left[
\ytableausetup{boxsize=1.2em, centertableaux}
\begin{ytableau}
a \\  
b & d\\  
c 
\end{ytableau}
\right]
\quad \overset{\Delta}{\longmapsto}
\left[
\ytableausetup{boxsize=1.2em, centertableaux}
\begin{ytableau}
\none &a \\  
\none &b & d\\  
c 
\end{ytableau}
\right]
-
\left[
\ytableausetup{boxsize=1.2em, centertableaux}
\begin{ytableau}
\none &a \\  
\none &c & d\\  
b 
\end{ytableau}
\right]
+
\left[
\ytableausetup{boxsize=1.2em, centertableaux}
\begin{ytableau}
\none&b \\  
\none&c & d\\  
a 
\end{ytableau}
\right].
$$
Meanwhile, viewing the image of $m$ as lying
in $S^{\rows(D \odot D')}=S^3(V) \otimes S^1(V)$, it sends
$$
\left[
\ytableausetup{boxsize=1.2em, centertableaux}
\begin{ytableau}
\none &w \\  
\none &x & y\\  
z
\end{ytableau}
\right]
\quad \overset{m}{\longmapsto}
zxy \otimes w - zwy \otimes x.
$$
Then their composite $m \circ \Delta$ is the zero map, as shown here:
$$
\begin{aligned}
\left[
\ytableausetup{boxsize=1.2em, centertableaux}
\begin{ytableau}
a \\  
b & d\\  
c 
\end{ytableau}
\right]
& \overset{m \circ \Delta}{\longmapsto}
m \left[
\ytableausetup{boxsize=1.2em, centertableaux}
\begin{ytableau}
\none &a \\  
\none &b & d\\  
c 
\end{ytableau}
\right]
-
m \left[
\ytableausetup{boxsize=1.2em, centertableaux}
\begin{ytableau}
\none &a \\  
\none &c & d\\  
b 
\end{ytableau}
\right]
+
m \left[
\ytableausetup{boxsize=1.2em, centertableaux}
\begin{ytableau}
\none&b \\  
\none&c & d\\  
a 
\end{ytableau}
\right]\\
&=cbd \otimes a - cad \otimes b
- ( bcd \otimes a - bad \otimes c )
+ acd \otimes b - abd \otimes c
=0.
\end{aligned}
$$
However one can check that in this case, the short complex \eqref{non-exact-short-complex} is not short exact, for example
by comparing their $GL(V)$-characters:
$$
\begin{aligned}
\ch_{
\bbs^{\ytableausetup{boxsize=0.4em}
\begin{ytableau} 
\, \\
 &  \\
\, \\
\end{ytableau}}} &=s_{(2,1,1)},\\
\ch_{
\bbs^{\ytableausetup{boxsize=0.4em}
\begin{ytableau} 
\none& \\
\none & & \\
\,\\
\end{ytableau}}} &= s_{(2,1,1)}+s_{(2,2)}+s_{(3,1)}+s_{(2,2)},\\
\ch_{
\bbs^{
\ytableausetup{boxsize=0.4em}
\begin{ytableau} 
\none& \\
 & & \\ 
\end{ytableau}}} & = s_{(3,1)}.
\end{aligned}
$$
The reader can also check that, replacing $D'$ above
with the diagram 
$
\ytableausetup{boxsize=0.8em}
\begin{ytableau}
\, & \\
x' \\
\end{ytableau}
$ in which one swaps the rows to make it a Ferrers diagram, the resulting complex would be short exact:
\begin{equation*}
0 \rightarrow 
\bbs^{
\ytableausetup{boxsize=0.4em}
\begin{ytableau} 
\,& \\
\,  \\
\,
\end{ytableau}
}
\overset{\Delta}{\longrightarrow}
\bbs^{\ytableausetup{boxsize=0.4em}
\begin{ytableau} 
\none& &\\
\none & \\
\,\\
\end{ytableau}}
\overset{m}{\longrightarrow}
\bbs^{\ytableausetup{boxsize=0.4em}
\begin{ytableau} 
\none& &\\
 &  \\
\end{ytableau}}
\rightarrow
0.
\end{equation*}
\end{example}

\begin{proof}[Proof of Proposition~\ref{skew-concatenation-near-concatenation-prop}]
We prove each part separately.
\vskip.1in
\noindent
{\sf Proof of (i).}
Let the last column of $D$ and first column of $D'$ contain $\ell, \ell'$ cells, respectively,
so that they are merged to form a column with $\ell+\ell'$ cells in $D \cdot D'$.  Then the factorization of the antisymmetrization map
$$
\wedge^{\ell+\ell'}(V) \rightarrow \wedge^{\ell}(V) \otimes \wedge^{\ell'}(V)
\rightarrow T^{\ell}(V) \otimes T^{\ell'}(V) 
=T^{\ell+\ell'}(V)
$$
leads to a factorization
of the antisymmetrization map 
$\fancyA_{D \cdot D'}$ through $\fancyA_{D \oplus D'}$ as follows:
$$
\wedge^{\cols{D \cdot D'}}(V) \rightarrow \wedge^{\cols{D \oplus D'}}(V)
\rightarrow T^{|D|+|D'|}(V).
$$
Since $\fancyB_{D \cdot D'}=\fancyB_{D \oplus D'}$,
this leads to an inclusion of the
Schur functors 
$$
\schurfunctor^{D \cdot D'} = 
\im(\fancyB_{D \cdot D'} \circ
\fancyA_{D \cdot D'})
\hookrightarrow 
\im(\fancyB_{D \oplus D'} \circ
\fancyA_{D \oplus D'})
=\schurfunctor^{D \oplus D'}.
$$

\vskip.1in
\noindent
{\sf Proof of (ii).}
Similarly, let the top row of $D$ and bottom row of $D'$ contain $r, r'$ cells, respectively,
so that they are merged to form a row with $r+r'$ cells in $D \odot D'$.  Then the factorization of the symmetrization map
$$
T^{r+r'}(V)=T^r(V) \otimes T^{r'}(V)  \rightarrow S^{r}(V) \otimes S^{r'}(V)
\rightarrow S^{r+r'}(V)
$$
leads to a factorization
of the symmetrization map 
$\fancyB_{D \odot D'}$
through
$\fancyB_{D \oplus D'}$ as follows:
$$
T^{|D|+|D'|}(V) \rightarrow S^{\rows{D \oplus D'}}(V) \rightarrow S^{\rows{D \odot D'}}(V)
$$
Since $\fancyA_{D \odot D'}=\fancyA_{D \oplus D'}$,
this leads to a surjection of the
Schur functors 
$$
\schurfunctor^{D \oplus D'} =
\im(\fancyB_{D \oplus D'} \circ
\fancyA_{D \oplus D'})
\twoheadrightarrow 
\im(\fancyB_{D \odot D'} \circ
\fancyA_{D \odot D'})
=\schurfunctor^{D \odot D'}.
$$

\vskip.1in
\noindent
{\sf Proof of (iii).}
We wish to show that when $D, D'$ have northeastmost, southwestmost cells $x,x'$ then the map 
$m \circ \Delta: 
\schurfunctor^{D \cdot D'} 
\rightarrow 
\schurfunctor^{D \odot D'}
$
is zero. Let 
$
N:=|D|+|D'|
$
denote the number of cells in any of the diagrams
$D \cdot D', D \oplus D', D \odot D'$.
One can express $m \circ \Delta$
as a composite
\begin{equation}
\label{m-delta-as-first-composite}
\wedge^{\cols(D \cdot D')}(V) 
\overset{\fancyA_{D \cdot D'}}{\longrightarrow} 
T^N(V)
\overset{\fancyB_{D \odot D'}}{\longrightarrow} 
S^{\rows(D \odot D')}(V).
\end{equation}
as illustrated in the following example.  Let 
$$
D=
\ytableausetup{boxsize=0.8em}
\begin{ytableau} 
\, & & x \\ 
 &\none & \\  
\end{ytableau}
\qquad
D'=
\ytableausetup{boxsize=1.0em}
\begin{ytableau} 
\none & \none & \\ 
 &  \\
x' & & 
\end{ytableau}
$$
so $N=|D|+|D'|=11$, and use these
consistent cell labelings in $D \cdot D'$ and $D\odot D'$, 
with $x=5, x'=6$ darkened, in bijection with
the tensor positions in $T^N(V)$: 
$$
D \cdot D'=
\ytableausetup{boxsize=1.0em}
\begin{ytableau} 
\none & \none& \none & \none & 11 \\ 
\none & \none& 7 & 9 \\
\none &\none & \mathbf{6} & 8 & 10 \\
2 & 3 & \mathbf{5} \\ 
1 &\none & 4 \\  
\end{ytableau}
%\qquad
%D \oplus D'=
%\ytableausetup{boxsize=1.0em}
%\begin{ytableau} 
%\none & \none& \none & \none&  \none & 11 \\ 
%\none & \none& \none &7 & 9 \\
%\none &\none &\none & \mathbf{6} & 8 & 10 \\
%2 & 3 & \mathbf{5} \\ 
%1 &\none & 4 \\  
%\end{ytableau}
\qquad \qquad
D \odot D'=
\ytableausetup{boxsize=1.0em}
\begin{ytableau} 
\none & \none& \none & \none & \none & 11 \\ 
\none & \none& \none & 7 & 9 \\
2 & 3 & \mathbf{5} & \mathbf{6} & 8 & 10 \\
1 &\none & 4 \\  
\end{ytableau}
$$
Then the map 
$\wedge^{\cols(D \cdot D')}(V) 
\overset{\fancyA_{D \cdot D'}}{\longrightarrow} 
T^N(V)$ sends
$$
(v_1 \wedge v_2) \otimes
v_3 \otimes
(v_4 \wedge v_5 \wedge v_6 \wedge v_7) \otimes
(v_8 \wedge v_9) \otimes
(v_{10} \wedge v_{11})
\longmapsto
(v_1 \otimes v_2 \otimes \cdots \otimes v_{10} \otimes v_{11}) \cdot \gamma^-_{D \cdot D'}
$$
where $\gamma^-_{D \cdot D'}$ is a particular column-antisymmetrizing element of the group
algebra $\kk[\symm_n]$ acting (on the right) on the tensor positions
in $T^N(V)$.  Specifically, 
$$
\gamma^-_{D \cdot D'}:=
\sum_{w} \sgn(w) \cdot w
$$
summing over $w$ in the group
$\symm_{\cols(D \cdot D')}$ of column permutations of $D \cdot D'$.
In this example, 
$$
\symm_{\cols(D \cdot D')}=\symm_{\{1,2\}} \times \symm_{\{3\}}
\times \symm_{\{4,5,6,7\}} \times \symm_{\{8,9\}} \times \symm_{\{10,11\}}.
$$
Meanwhile, the map 
$T^N(V)
\overset{\fancyB_{D \odot D'}}{\longrightarrow} 
S^{\rows(D \odot D')}(V)$
sends
$$
u_1 \otimes u_2 \otimes \cdots \otimes u_{10} \otimes u_{11}
\longmapsto \,\, u_1 u_4 \,\, \otimes 
\,\, u_2 u_3 u_5 u_6 u_8 u_{10}\,\, \otimes 
\,\, u_7 u_9 \,\, \otimes 
\,\, u_{11}. 
$$

Now note that $x,x'$ lie in the same column of $D \cdot D'$.
Thus in $\kk[\symm_n]$, one can factor
$$
\gamma^-_{D \cdot D'} = \epsilon \cdot (1-t_{x,x'})
$$
where $t_{x,x'}$ is the transposition of the labels
on the cells $x,x'$ (so $t_{x,x'}=t_{5,6}$ in the above example), and $\epsilon:=\sum_w \sgn(w) \cdot w$ where $w$ ranges over
any choice of coset representatives for the cosets $\symm_{\cols(D \cdot D')}/\langle t_{x,x'}\rangle$.   Hence the image of 
$\fancyA_{D \cdot D'}$ lies in $T^N(V) \cdot (1-t_{x,x'})$.

On the other hand, since $x,x'$ lie in the same row of $D \odot D'$, one
can check that  
$\fancyB_{D \odot D'}$ annihilates all of $T^N(V) \cdot (1-t_{x,x'})$.  Hence the composite $m \circ \Delta=\fancyB_{D \odot D'}\circ \fancyA_{D \cdot D'}=0$.

\vskip.1in
\noindent
{\sf Proof of (iv).}
When $D, D'$ are both skew shapes, so are all three of $D \oplus D', D \cdot D', D \odot D'$,
and hence all three Schur functors $s_{D \oplus D'},s_{D \cdot D'}, s_{D \odot D'}$ are universally free by \cite[Theorem II.2.16]{akin1982schur}.  Then Proposition~\ref{skew-concatenation-near-concatenation-identity} shows that for any coefficient ring $\kk$, the short complex in part (iii) is one of free $\kk$-modules of the form
$
0 \longrightarrow \kk^a
\overset{\Delta}{\longrightarrow} \kk^{a+b}
\overset{m}{\longrightarrow} \kk^b
\longrightarrow 0
$
for fixed integers $a,b$ independent of $\kk$.

Consequently, whenever $\kk$ is a {\it field}, this complex is short exact by dimension-counting.

We claim this implies it must also be exact when $\kk=\ZZ$:  
one has an inclusion of free $\ZZ$-modules
$\im(\Delta) \subseteq \ker(m)$ inside $\ZZ^{a+b}$, both of same rank $a$, 
and if any prime $p$ were to divide the index
$[\ker(m):\im(\Delta)]$, it would contradict exactness when $\kk=\mathbb{F}_p$.

Lastly, once one knows that for $\kk=\ZZ$ the sequence is
short exact, it must also split,
since $\ZZ^b$ is a projective $\ZZ$-module.
Then applying $(-) \otimes_\ZZ \kk$ to this split short exact sequence of free $\ZZ$-modules shows that the sequence is split short exact over any coefficient ring $\kk$.
\end{proof}

\begin{remark}
The short complexes and short exact sequences in parts (iii),(iv) of Proposition~\ref{skew-concatenation-near-concatenation-prop} are closely related to special cases
of short complexes of Specht modules for $\symm_n$ considered by James and Peel in \cite{JamesPeel}; see also Liu \cite[\S 2.2]{LiuThesis}. They are also dual to special cases of short complexes of {\it Weyl modules} for $GL(V)$ considered by Shimozono in \cite{Shimozono}.
\end{remark}

\subsection{Resolving concatenations by near-concatenations}

Note that both operations 
$\alpha \cdot \beta$ and $\alpha \odot \beta$ are associative, and that they associate with each other, meaning
$$
\begin{aligned}
\alpha \cdot (\beta \odot \gamma)& =
(\alpha \cdot \beta) \odot \gamma,\\
\alpha \odot (\beta \cdot \gamma)& =
(\alpha \odot \beta) \cdot \gamma.\\
\end{aligned}
$$
Thus one can write sequences of these
operations unambiguously without parentheses.
In particular, for any sequence
$\underline{\alpha} = (\alpha^{(1)} , \dots , \alpha^{(\ell)})$ of compositions,
one can define the compositions
$
\alpha^{(1)} \cdot \alpha^{(2)}  \cdots  \alpha^{(\ell)}
$
and 
$
\alpha^{(1)} \odot \alpha^{(2)} \odot \cdots \odot \alpha^{(\ell)}.
$

We next note that one may recast the basic ribbon Schur function identity \eqref{concatenation-near-concatenation-identity} in the form
$$
s_{\alpha \cdot \beta} =
s_{\alpha} s_{\beta} - s_{\alpha \odot \beta}
=\det\left[
\begin{matrix}
s_\alpha & s_{\alpha \odot \beta} \\
1 & s_\beta
\end{matrix}
\right]
$$
and regard it as the $\ell=2$ special case
of the following family of identities.

\begin{prop}
\label{horizontal-ribbon-Hamel-Goulden-prop}
For any sequence of compositions 
$\underline{\alpha} = (\alpha^{(1)} , \dots , \alpha^{(\ell)})$, one has 
$$
s_{\alpha^{(1)} \cdot 
\alpha^{(2)} \cdots 
\alpha^{(\ell)}  }
=\det \left[ H_{ij} \right]_{i,j=1}^\ell
\quad \text{ where } \quad
H_{i,j}:=
\begin{cases}
0 & \text{ if }i \geq j+2,\\
1 & \text{ if }i = j+1,\\
s_{\alpha^{(i)} \odot \alpha^{(i+1)} \odot \cdots \odot \alpha^{(j)}} 
& \text{ if }i \leq j.
\end{cases}
$$
\end{prop}
\begin{example}
The determinantal identity in the proposition looks as follows for $\ell=3, 4$:
\begin{equation}
\label{HG-identity-for-n=3}
\begin{aligned}
s_{\alpha \cdot \beta \cdot \gamma}
&=\det\left[ 
\begin{matrix}
s_\alpha & s_{\alpha \odot \beta} &  s_{\alpha \odot \beta \odot \gamma}\\
1 & s_\beta & s_{\beta \odot \gamma} \\
0 & 1 & s_\gamma
\end{matrix}
\right] \\
&= s_\alpha s_\beta s_\gamma
-\left(\begin{matrix} s_{\alpha \odot \beta} s_\gamma  \\
+ s_\alpha s_{\beta \odot \gamma} 
\end{matrix} \right)
+ s_{\alpha \odot \beta \odot \gamma}.
\end{aligned}
\end{equation}

\begin{equation}
\label{HG-identity-for-n=4}
\begin{aligned}
s_{\alpha \cdot \beta \cdot \gamma \cdot \delta}
&=\det\left[ 
\begin{matrix}
s_\alpha & s_{\alpha \odot \beta} &  s_{\alpha \odot \beta \odot \gamma}&  s_{\alpha \odot \beta \odot \gamma \odot \delta}\\
1 & s_\beta & s_{\beta \odot \gamma} & s_{\beta \odot \gamma \odot \delta} \\
0 & 1 & s_\gamma & s_{\gamma\odot \delta} \\
0 & 0 & 1 & s_\delta
\end{matrix}
\right]\\
&= s_\alpha s_\beta s_\gamma s_\delta
-\left(\begin{matrix} s_{\alpha \odot \beta} s_\gamma s_\delta  \\
+ s_\alpha s_{\beta \odot \gamma} s_\delta\\
+s_\alpha s_\beta s_{\gamma \odot \delta}
\end{matrix} \right)
+\left(\begin{matrix} s_{\alpha \odot \beta \odot \gamma} s_\delta  \\
+ s_{\alpha \odot \beta} s_{\gamma \odot \delta}\\
+s_\alpha s_{\beta \odot \gamma \odot \delta}
\end{matrix} \right)
- s_{\alpha \odot \beta \odot \gamma \odot \delta}.
\end{aligned}
\end{equation}
\end{example}
\begin{proof}
Given $\underline{\alpha} = (\alpha^{(1)} , \dots , \alpha^{(\ell)})$, denote the matrix $H$ in the proposition as $H(\underline{\alpha})$, and define
two sequences of compositions of length $\ell-1$:
$$
\begin{aligned}
\underline{\hat{\alpha}}
&:=(\alpha^{(2)},\alpha^{(3)},
\ldots,\alpha^{(\ell)})\\
\underline{\alpha'}
&:=(\alpha^{(1)} \odot \alpha^{(2)},\alpha^{(3)},
\ldots,\alpha^{(\ell)})
\end{aligned}
$$
The proposition follows by
induction on $\ell$, via Laplace expansion
in the first column:
$$
\begin{aligned}
\det H(\underline{\alpha}) 
&=s_{\alpha^{(1)}} \cdot \det H(\underline{\hat{\alpha}}) 
- \det H(\underline{\alpha}')\\
&=s_{\alpha^{(1)}} \cdot s_{\alpha^{(2)} \cdot \alpha^{(3)} \cdots \alpha^{(\ell)}}
- s_{ \alpha^{(1)} \odot \alpha^{(2)} \cdot \alpha^{(3)} \cdots \alpha^{(\ell)}}\\
&\overset{(*)}{=}\left( 
s_{\alpha^{(1)} \cdot \alpha^{(2)} \cdot \alpha^{(3)} \cdots \alpha^{(\ell)}}
+
s_{\alpha^{(1)} \odot \alpha^{(2)} \cdot \alpha^{(3)} \cdots \alpha^{(\ell)}}
\right)
- s_{ \alpha^{(1)} \odot \alpha^{(2)} \cdot \alpha^{(3)} \cdots \alpha^{(\ell)}}\\
&=s_{\alpha^{(1)} \cdot \alpha^{(2)} \cdot \alpha^{(3)} \cdots \alpha^{(\ell)}}.
\end{aligned}
$$
Here the equality marked $(*)$
used identity \eqref{concatenation-near-concatenation-identity}.
\end{proof}

\begin{remark}
Proposition~\ref{horizontal-ribbon-Hamel-Goulden-prop} is a very special case of the {\it Hamel-Goulden} determinantal identity \cite{HamelGoulden}, which for any skew diagram $D$ expresses the skew Schur function $s_D$ as a determinant, based on any decomposition of $D$ into ribbon subdiagrams
that they call a {\it planar outside decomposition}.  Proposition~\ref{horizontal-ribbon-Hamel-Goulden-prop} is the special case of the Hamel-Goulden identity in which 
$D=\alpha^{(1)} \cdot \alpha^{(2)} \cdots
\alpha^{(\ell)}$ and the planar outside decomposition of $D$ is into the individual ribbon subdiagrams $\alpha^{(i)}$ for $i=1,2,\ldots,\ell$ inside it.
\end{remark}

We next seek to lift Proposition~\ref{horizontal-ribbon-Hamel-Goulden-prop}
to a homological result, Theorem~\ref{lem:HGcatComplex}, generalizing
Proposition~\ref{concatenation-near-concatenation-prop}
in the case $\ell=2$, and using it as a base
case for induction on $\ell$.  To this end,
given any length $\ell$ sequence $\underline{\alpha}=(\alpha^{(1)},\ldots,\alpha^{(\ell)})$ of compositions, and for any choice of subset of indices $I \subseteq [\ell-1]:=\{1,2,\ldots,\ell-1\}$, let
$\underline{\alpha}(I)$ be the length $\ell-|I|$
sequence of compositions that replaces the $i^{th}$ comma in $\underline{\alpha}$ with the $\odot$ operation,
for each $i$ in $I$.  

\begin{example}
Let  $\ell=5$ and
$\underline{\alpha}=(\alpha,\beta,\gamma,\delta,\epsilon)$.  Then
$$
\begin{aligned}
\underline{\alpha}(\varnothing)&=(\alpha,\beta,\gamma,\delta,\epsilon)=\underline{\alpha},\\
\underline{\alpha}(\{3\})
&=(\alpha,\beta,\gamma \odot \delta,\epsilon)\\
\underline{\alpha}(\{1,3\})
&=(\alpha\odot\beta,\gamma \odot \delta,\epsilon)\\
\underline{\alpha}(\{2,3\})
&=(\alpha,\beta\odot\gamma \odot \delta,\epsilon)\\
\underline{\alpha}(\{1,2,3,4\})
&=(\alpha\odot\beta\odot\gamma \odot \delta\odot\epsilon)
\end{aligned}
$$
\end{example}
\noindent
Using this notation, and also an abbreviation for
$\underline{\alpha}=(\alpha^{(1)},\ldots,\alpha^{(\ell)})$ of the product
$$s_{\underline{\alpha}}:
=s_{\alpha^{(1)}} s_{\alpha^{(2)}} \cdots s_{\alpha^{(\ell)}},
$$
one can rephrase the determinant expansion in  Proposition~\ref{horizontal-ribbon-Hamel-Goulden-prop} as follows:
\begin{equation}
    \label{rephrased-HG-determinant}
s_{\alpha^{(1)} \cdot 
\alpha^{(2)} \cdots 
\alpha^{(\ell)}  }
=\det H
=\sum_{I \subseteq [\ell-1]}
(-1)^{|I|}
s_{\underline{\alpha}(I)}
=\sum_{i=0}^{\ell-1} (-1)^i
\left( 
\sum_{\substack{I \subseteq [\ell-1]:\\ |I|=i}} s_{\underline{\alpha}(I)}
\right).
\end{equation}

\begin{definition}\label{def: Hamel-Goulden-categorification}
\rm
Given 
$\underline{\alpha}=(\alpha^{(1)},\ldots,\alpha^{(\ell)})$, define
$$
\bbs^{\underline{\alpha}}  := \bbs^{\sigma ( \alpha^{(1)})}  \otimes_\kk \cdots \otimes_\kk \bbs^{\sigma (\alpha^{(\ell)}) } .
$$
Create a cochain complex $(\HHH(\underline{\alpha}),\delta)$
whose $i^{th}$ term is 
$$
\HHH_i(\underline{\alpha})
:=\bigoplus_{\substack{I \subseteq [\ell-1]:\\ |I|=i}} \bbs^{\underline{\alpha}(I)},
$$
modeling the $i^{th}$ inner summand
on the far right of \eqref{rephrased-HG-determinant}.
Define the differential 
$$
\begin{array}{ccc}
\HHH_i(\underline{\alpha}) &\overset{\delta_i}{\xrightarrow{\qquad\qquad}} &
\HHH_{i+1}(\underline{\alpha})\\
 \Vert& &\Vert\\
  \displaystyle\bigoplus_{\substack{I \subseteq [\ell-1]:\\ |I|=i}} \bbs^{\underline{\alpha}(I)}
  &  &
 \displaystyle\bigoplus_{\substack{J \subseteq [\ell-1]:\\ |J|=i+1}} \bbs^{\underline{\alpha}(J)}
\end{array}
$$
in block matrix form as follows.
The $(I,J)$-block of $\delta_i$ is zero unless $J \supset I$. If $J \supset I$, define
$
\sgn(I,J):=(-1)^m
$
where $J=\{j_0 < j_1 < \cdots < j_i\}$
and $J = I \sqcup \{j_m\}$.  Then define
the 
$(I,J)$-block of $\delta_i$ to be  
$\sgn(I,J) \cdot \nabla^{\underline{\alpha}(J)}_{\underline{\alpha}(I)}$, where 
$$
\nabla^{\underline{\alpha}(J)}_{\underline{\alpha}(I)}:
\bbs^{\underline{\alpha}(I)} 
\rightarrow \bbs^{\underline{\alpha}(J)} 
$$
is the identity in most tensor factors,
except that it is the surjective map $m_{\alpha^{(j_m)},\alpha^{(j_{m+1})}}$ from Proposition~\ref{concatenation-near-concatenation-prop} sending the two tensor factors of 
$\bbs^{\underline{\alpha}(I)} $
that involve $\alpha^{(j_m)}, \alpha^{(j_m+1)}$ 
onto the unique tensor factor of 
$\bbs^{\underline{\alpha}(J)} $
involving $\alpha^{(j_m)} \odot \alpha^{(j_m+1)}$.

\end{definition}

The following homological lift of
Proposition~\ref{horizontal-ribbon-Hamel-Goulden-prop} or \eqref{rephrased-HG-determinant} is the main result of this section.

\begin{thm}
\label{lem:HGcatComplex}
For any sequence $\underline{\alpha}$, one has that $(\HHH(\underline{\alpha}),\delta)$ is a cochain complex, with
$$
H^0(\HHH(\underline{\alpha})) \cong
\schurfunctor^{\alpha^{(1)} \cdot \alpha^{(2)} \cdots \alpha^{(\ell)} }
$$
and which is acyclic in strictly 
positive homological degrees.
\end{thm}

In other words, $\HHH(\underline{\alpha})$ gives a (right or co-)resolution of 
the $GL(V)$-module $\schurfunctor^{\alpha^{(1)} \cdot \alpha^{(2)} \cdots \alpha^{(\ell)} }$:
$$
0 \to 
\schurfunctor^{\alpha^{(1)} \cdot \alpha^{(2)} \cdots \alpha^{(\ell)} }
\to
\HHH(\underline{\alpha})
\to 0.
$$

\begin{example}
With notation $\otimes_\kk =:\otimes$, the $n=3$ lifting of \eqref{HG-identity-for-n=3} looks as follows:
\begin{equation}
\label{lift-of-HG-for-n=3}
	0 \to 
	\bbs^{\sigma ( \alpha , \beta , \gamma)} 
	\to 
	\bbs^{\sigma (\alpha)}  \otimes \bbs^{\sigma (\beta)}  \otimes \bbs^{\sigma (\gamma)} 
	\xrightarrow[]{\left[\begin{matrix} \nabla^{\{1\}}_{\varnothing} \\ \nabla^{\{2\}}_{\varnothing} \end{matrix}\right]}
	\begin{matrix}
	\bbs^{\sigma (\alpha \odot \beta)}  \otimes \bbs^{\sigma (\gamma)} \\
	\oplus\\
	\bbs^{\sigma (\alpha)} \otimes \bbs^{\sigma (\beta \odot \gamma)}
	\end{matrix}
	\xrightarrow{\left[\begin{matrix} \nabla^{\{1,2\}}_{\{1\}} & \nabla^{\{1,2\}}_{\{2\}} \end{matrix}\right]}
	\bbs^{\sigma (\alpha \odot \beta \odot \gamma)}
	\to 0
\end{equation}
%$$
%\begin{tikzcd}[sep=small]
%	& & & \bbs^{\sigma (\alpha)} \otimes \bbs^{\sigma (\beta \odot \gamma)} \arrow[dr]& &  \\
%	0 \arrow[r]& \bbs^{\sigma ( \alpha , \beta , \gamma)} \arrow[r]  & %\bbs^{\sigma (\alpha)}  \otimes \bbs^{\sigma (\beta)}  \otimes \bbs^{\sigma (\gamma)}\arrow[ur] \arrow[dr] & &
%	\bbs^{\sigma (\alpha \odot \beta \odot \gamma)}\arrow[r] & 0\\
%& & &\bbs^{\sigma (\alpha \odot \beta)}  \otimes \bbs^{\sigma (\gamma)} \arrow[ur] & &  \\
%\end{tikzcd}
%$$
\end{example}

\begin{proof}[Proof of Theorem~\ref{lem:HGcatComplex}]
To see that $\HHH(\underline{\alpha})$ is a complex, the definition of $\sgn(I,J)$ reduces to checking that for any $j,j' \not\in I$ one has
commutativity of the appropriate $\nabla$ maps in this diamond:
$$
  \begin{tikzcd}[sep=small]
  %[column sep =small] 
   & \schurfunctor^{\underline{\alpha}(I)}
   \arrow{dl}\arrow[rightarrow]{dr}
     &   \\  
   \schurfunctor^{\underline{\alpha}(I \sqcup\{j\})} \arrow{dr} 
   & & \schurfunctor^{\underline{\alpha}(I\sqcup\{j'\})} \arrow{dl} \\
   & \schurfunctor^{\underline{\alpha}(I\sqcup\{j,j'\})}
     & 
  \end{tikzcd}
$$
This commutativity is easy when $|j'-j| \geq 2$. For $|j'-j|=1$, it follows from
commutativity of the appropriate surjections
$m_{\alpha,\beta}$ from Proposition~\ref{concatenation-near-concatenation-prop} appearing here:
\begin{equation}
    \label{important-commutativity-diamond}
  \begin{tikzcd}[column sep=small]
   & \schurfunctor^{\sigma(\alpha)}
   \otimes_\kk \schurfunctor^{\sigma(\beta)}
   \otimes_\kk \schurfunctor^{\sigma(\gamma)}
   \arrow{dl}\arrow{dr}
     &   \\  
   \schurfunctor^{\sigma(\alpha \odot \beta)}
   \otimes_\kk \schurfunctor^{\sigma(\gamma)} \arrow{dr} 
   & &
    \schurfunctor^{\sigma(\alpha)}
   \otimes_\kk \schurfunctor^{\sigma(\beta \odot \gamma)}
   \arrow{dl}\\
   & \schurfunctor^{\sigma(\alpha \odot \beta \odot \gamma)}
     & 
  \end{tikzcd}
\end{equation}

%Let $\alpha$, $\beta$, and $\gamma$ be compositions. Then the following diagram commutes:
% https://q.uiver.app/?q=WzAsNCxbMCwwLCJcXGJic157XFxzaWdtYSAoXFxhbHBoYSl9IChWKSBcXG90aW1lc19cXGtrIFxcYmJzXntcXHNpZ21hIChcXGJldGEpfSAoVikgXFxvdGltZXNfXFxrayBcXGJic157XFxzaWdtYSAoXFxnYW1tYSl9IChWKSJdLFsyLDAsIlxcYmJzXntcXHNpZ21hIChcXGFscGhhIFxcb2RvdCBcXGJldGEpfSBcXG90aW1lc19cXGtrIFxcYmJzXntcXHNpZ21hKCBcXGdhbW1hKX0gKFYpIl0sWzAsMSwiXFxiYnNee1xcc2lnbWEgKFxcYWxwaGEpfSAoVikgXFxvdGltZXNfXFxrayBcXGJic157XFxzaWdtYSAoIFxcYmV0YSBcXG9kb3QgXFxnYW1tYSl9IChWKSJdLFsyLDEsIlxcYmJzXntcXHNpZ21hIChcXGFscGhhIFxcb2RvdCBcXGJldGEgXFxvZG90IFxcZ2FtbWEpfSAoVikiXSxbMCwxXSxbMiwzXSxbMSwzXSxbMCwyXV0=
%\[\begin{tikzcd}
%	{\bbs^{\sigma (\alpha)} (V) \otimes_\kk \bbs^{\sigma (\beta)} (V) \otimes_\kk \bbs^{\sigma (\gamma)} (V)} && {\bbs^{\sigma (\alpha \odot \beta)} \otimes_\kk \bbs^{\sigma( \gamma)} (V)} \\
%	{\bbs^{\sigma (\alpha)} (V) \otimes_\kk \bbs^{\sigma ( \beta \odot \gamma)} (V)} && {\bbs^{\sigma (\alpha \odot \beta \odot \gamma)} (V)}
%	\arrow[from=1-1, to=1-3]
%	\arrow[from=2-1, to=2-3]
%	\arrow[from=1-3, to=2-3]
%	\arrow[from=1-1, to=2-1]
%\end{tikzcd}\]
%where all maps in the above diagram are induced by the respective multiplication in the symmetric algebra. 

For the remaining assertions, proceed by induction on $\ell$, with base case given by the case $\ell = 2$, which one can check is Proposition \ref{concatenation-near-concatenation-prop}. 
In the inductive step, recall from the proof of
Proposition~\ref{horizontal-ribbon-Hamel-Goulden-prop}
these shorter sequences of compositions
$$
\begin{aligned}
\underline{\hat{\alpha}}
&:=(\alpha^{(2)},\alpha^{(3)},
\ldots,\alpha^{(\ell)})\\
\underline{\alpha'}
&:=(\alpha^{(1)} \odot \alpha^{(2)},\alpha^{(3)},
\ldots,\alpha^{(\ell)})
\end{aligned}
$$
whose associated matrices $H(\underline{\alpha})$
satisfied an identity, coming from Laplace expansion:
$$
\det H(\underline{\alpha}) 
=s_{\alpha^{(1)}} \cdot \det H(\underline{\hat{\alpha}}) 
- \det H(\underline{\alpha}').
$$
We next lift this relation homologically.
Define a morphism of complexes
$$
\Phi : 
\bbs^{\sigma (\alpha^{(1)})} \otimes_\kk
 \HHH( \underline{\hat{\alpha}} )
 \to 
\HHH(\underline{\alpha}')
$$
by having it map between their terms indexed by $I \subseteq [\ell-2]$ as follows:
$$
\nabla_{(\alpha^{(1)} , \underline{\hat{\alpha} }(I))}^{\underline{\alpha' }(I)}:
\bbs^{\sigma (\alpha^{(1)})}  \otimes_\kk \bbs^{\underline{\hat{\alpha}}(I)} 
\longrightarrow
\bbs^{\underline{\alpha}'(I)}
$$
Checking $\Phi$ is a morphism of complexes means checking for $j \not\in I$ that
this commutes:
$$
\begin{tikzcd}[column sep =small] 
\bbs^{\sigma (\alpha^{(1)})} \otimes_\kk \bbs^{\underline{\hat{\alpha}}(I)} \arrow{rr}{\nabla} \arrow{d}{1 \otimes \nabla} & &
\bbs^{\underline{\alpha}'(I)}\arrow{d}{1 \otimes \nabla} \\
\bbs^{\sigma (\alpha^{(1)})} \otimes_\kk \bbs^{\underline{\hat{\alpha}}(I \sqcup \{j\})} \arrow{rr}{\nabla} & &
\bbs^{\underline{\alpha}'(I \sqcup \{j\})} \\
 \end{tikzcd}
$$
This is easy for $j \geq 2$, and for $j=1$
is commutativity of \eqref{important-commutativity-diamond}
with $(\alpha,\beta,\gamma)=(\alpha^{(1)},\alpha^{(2)},\alpha^{(3)})$. 

%one can appeal to the fact that $P (\alpha^{(1)} , \dots , \alpha^{(\ell)})$ is a thin poset to restrict to showing commutativity of diagrams of the form
% https://q.uiver.app/?q=WzAsNCxbMSwwLCJcXGJic157XFx1bmRlcmxpbmV7XFxhbHBoYShJKX19IChWKSJdLFswLDEsIlxcYmJzXntcXHVuZGVybGluZXtcXGFscGhhKEkgXFxiYWNrc2xhc2ggaSl9fSAoVikiXSxbMiwxLCJcXGJic157XFx1bmRlcmxpbmV7XFxhbHBoYShJIFxcYmFja3NsYXNoIGopfX0gKFYpIl0sWzEsMiwiXFxiYnNee1xcdW5kZXJsaW5le1xcYWxwaGEoSSBcXGJhY2tzbGFzaCBpLGopfX0gKFYpIl0sWzAsMV0sWzAsMl0sWzEsM10sWzIsM11d
%\[\begin{tikzcd}
%	& {\bbs^{\underline{\alpha}(I)} (V)} \\
%	{\bbs^{\underline{\alpha}(I \setminus \{i\})} (V)} && {\bbs^{\underline{\alpha}(I \setminus \{j\})} (V)} \\
%	& {\bbs^{\underline{\alpha}(I \setminus \{i,j\})} (V)}
%	\arrow[from=1-2, to=2-1]
%	\arrow[from=1-2, to=2-3]
%	\arrow[from=2-1, to=3-2]
%	\arrow[from=2-3, to=3-2]
%\end{tikzcd}\]
%where $I \subset [\ell-1]$ and $i,j \in I$ in the above diagram. But this is precisely the diagram of Observation \ref{obs:coassocDiagram}, in which case $\Phi$ is evidently a morphism of complexes. 

By construction one has an isomorphism between $\HHH(\underline{\alpha})$ and  
the {\it mapping cone} of $\Phi$:
$$
\cone (\Phi)  \cong \HHH(\underline{\alpha}).
$$
Hence there is a short exact sequence of cochain complexes
$$
\begin{array}{rcccccl}
0 \to 
&\HHH(\underline{\alpha}')[-1]
&\to &\cone (\Phi)& \to & 
\bbs^{\sigma (\alpha^{(1)})} (V) \otimes_\kk \HHH(\underline{\hat{\alpha}} )& 
\to 0.\\
 &\Vert&&\Vert&&\Vert& \\
 &A&&B&&C
\end{array}
$$
leading to a long exact sequence in cohomology:
\begin{equation}
    \label{LES-in-cohomology}
0 \to H^0 A \to H^0 B \to H^0 C
\to H^1 A \to H^1 B \to H^1 C
\to \cdots
\end{equation}
Applying the inductive hypothesis for $\HHH(\underline{\alpha}'), \HHH(\underline{\hat{\alpha}})$
and 
Proposition~\ref{tensoring-with-skews-is-exact-prop} to the subsequences
$$
\begin{array}{ccccc}
H^i A &\to& H^i B &\to& H^i C\\
\Vert& &\Vert& &\Vert\\
H^{i-1} \HHH(\underline{\alpha}')
& &H^i \HHH(\underline{\alpha})& 
& H^i( \schurfunctor^{\alpha^{(1)}} \otimes_\kk \HHH(\underline{\hat{\alpha}}))\\
\Vert& & & &\Vert\\
0& & & & \schurfunctor^{\alpha^{(1)}} \otimes_\kk  H^i \HHH(\underline{\hat{\alpha}}) \\
& & & &\Vert\\
& & & & 0
\end{array}
$$
shows that
$H^i\HHH(\underline{\alpha})=0$ 
for $i \geq 2$.  On the other hand,
the outer terms in \eqref{LES-in-cohomology} also vanish:
$$
\begin{aligned}
H^0A& = H^{-1}\HHH(\underline{\alpha}') =0,\\
H^1C&=H^1 ( \schurfunctor^{\alpha^{(1)}} \otimes_\kk \HHH(\underline{\hat{\alpha}}))
= \schurfunctor^{\alpha^{(1)}} \otimes_\kk  H^1 \HHH(\underline{\hat{\alpha}})=0
\end{aligned}
$$
where the second vanishing again used the inductive hypothesis
for $\HHH(\underline{\hat{\alpha}})$
and 
Proposition~\ref{tensoring-with-skews-is-exact-prop}.
Thus the long exact sequence starts
$$
\begin{array}{rcccccccl}
0 \to& H^0 B &\to& H^0 C
&\to& H^1 A &\to& H^1 B &\to 0\\
 &\Vert& &\Vert& &\Vert& &\Vert&  \\
0 \to &H^0 \HHH(\underline{\alpha})
&\to&H^0( \schurfunctor^{\alpha^{(1)}} \otimes_\kk \HHH(\underline{\hat{\alpha}}))
&\to&H^0 \HHH(\underline{\alpha}')
&\to&H^1 \HHH(\underline{\alpha})
&\to 0\\
 & & &\Vert& &\Vert& & &  \\
 & & & 
\schurfunctor^{\alpha^{(1)}} \otimes_\kk \schurfunctor^{\alpha^{(2)} \cdots \alpha^{(\ell)}}  
&\overset{f}{\to}&
\schurfunctor^{\alpha^{(1)} \odot \alpha^{(2)} \cdots \alpha^{(\ell)}}
& & & 
\end{array}
$$
where the bottommost equalities
used the inductive hypothesis on $H^0$
applied to 
$\HHH(\underline{\alpha}'), \HHH(\underline{\hat{\alpha}})$.

One can then check that the map labeled $f$ above, which is a connecting homomorphism in the long exact sequence, actually coincides with the surjective map $m_{\alpha,\beta}$ 
in this instance of the short 
exact sequence of Proposition \ref{concatenation-near-concatenation-prop} with
$\alpha=\alpha^{(1)}$ and
$\beta = \alpha^{(2)} \cdot \alpha^{(3)} \cdots \alpha^{(\ell)}$:
$$
0 \to
\bbs^{\sigma (\alpha^{(1)}  \cdot 
\alpha^{(2)} \cdots \alpha^{(\ell)})} 
\overset{\Delta_{\alpha,\beta}}{\longrightarrow}
\schurfunctor^{\alpha^{(1)}} \otimes_\kk \schurfunctor^{\alpha^{(2)} \cdots \alpha^{(\ell)}}  
\overset{m_{\alpha,\beta}}{\longrightarrow}
\schurfunctor^{\alpha^{(1)} \odot \alpha^{(2)} \cdots \alpha^{(\ell)}}
\to 0
$$
This gives the last two desired conclusions of the theorem:
$$
\begin{array}{rcccccl}
H^1 \HHH(\underline{\alpha})
&=&\coker(f) &\cong& \coker(m)&=& 0,\\
H^0 \HHH(\underline{\alpha})
&=&\ker(f)  &\cong& \ker(m)&\cong& \bbs^{\sigma (\alpha^{(1)}  \cdot 
\alpha^{(2)} \cdots \alpha^{(\ell)})} .\qedhere
\end{array}
$$
\end{proof}

The following result arises from studying the $\ell=3$ case of the complex introduced in Definition \ref{def: Hamel-Goulden-categorification}. The utility of the following proposition will become apparent in the $\Tor$ computations in Sections \ref{resolution-subsection} and \ref{tor-between-the-modules-subsection}.

\begin{prop}\label{prop:intersectingSchurs}
For any three compositions $(\alpha, \beta, \gamma)$ one has an equality
$$
\bbs^{\sigma(\alpha,\beta,\gamma)} =
\left( \bbs^{\sigma (\alpha , \beta)}  \otimes_\kk \bbs^{\sigma(\gamma)}  \right) \cap \left( \bbs^{\sigma (\alpha)}  \otimes_\kk \bbs^{\sigma(\beta,\gamma)}  \right),
$$
where the intersection is taken inside 
$$
\bbs^{\sigma (\alpha)} 
\otimes_\kk 
\bbs^{\sigma(\beta)} 
\otimes_\kk 
\bbs^{\sigma(\gamma)}.
$$ 
\end{prop}

\begin{proof}
%The augmented complex $R ( \alpha , \beta , \gamma)$ begins as:
%{\small % https://q.uiver.app/?q=WzAsNSxbMCwxLCIwIl0sWzEsMSwiXFxiYnNee1xcc2lnbWEgKCBcXGFscGhhICwgXFxiZXRhICwgXFxnYW1tYSl9IChWKSAiXSxbMiwxLCJcXGJic157XFxzaWdtYSAoXFxhbHBoYSl9IChWKSBcXG90aW1lc19cXGtrIFxcYmJzXntcXHNpZ21hIChcXGJldGEpfSAoVikgXFxvdGltZXNfXFxrayBcXGJic157XFxzaWdtYSAoXFxnYW1tYSl9IChWKSJdLFszLDAsIlxcYmJzXntcXHNpZ21hIChcXGFscGhhKX0gKFYpIFxcb3RpbWVzX1xca2sgXFxiYnNee1xcc2lnbWEgKFxcYmV0YSwgXFxnYW1tYSl9IChWKSJdLFszLDIsIlxcYmJzXntcXHNpZ21hIChcXGFscGhhLFxcYmV0YSl9IChWKSBcXG90aW1lc19cXGtrIFxcYmJzXntcXHNpZ21hIChcXGdhbW1hKX0gKFYpIl0sWzAsMV0sWzEsMl0sWzIsM10sWzIsNF1d
%\[\begin{tikzcd}
%	&&& {\bbs^{\sigma (\alpha)} (V) \otimes_\kk \bbs^{\sigma (\beta \odot \gamma)} (V)} \\	0 & {\bbs^{\sigma ( \alpha , \beta , \gamma)} (V) } & {\bbs^{\sigma (\alpha)} (V) \otimes_\kk \bbs^{\sigma (\beta)} (V) \otimes_\kk \bbs^{\sigma (\gamma)} (V)} \\	&&& {\bbs^{\sigma (\alpha \odot \beta)} (V) \otimes_\kk \bbs^{\sigma (\gamma)} (V)}
%	\arrow[from=2-1, to=2-2]
%	\arrow[from=2-2, to=2-3]
%	\arrow[from=2-3, to=1-4]
%	\arrow[from=2-3, to=3-4]
%\end{tikzcd}\]}

Exactness in the left and middle terms of
\eqref{lift-of-HG-for-n=3} shows that
$$
\bbs^{\sigma(\alpha,\beta,\gamma)} =
\ker( \nabla^{\{1\}}_{\varnothing} )
\cap \ker( \nabla^{\{2\}}_{\varnothing} ).
$$
Then Proposition~\ref{tensoring-with-skews-is-exact-prop}
together with exactness on the left in
Proposition~\ref{concatenation-near-concatenation-prop} identifies the kernels of
$\nabla^{\{1\}}_{\varnothing},
\nabla^{\{2\}}_{\varnothing}$
with the terms intersected on the right
in the proposition.
%there are also equalities
%{\small $$ \bbs^{\sigma (\alpha)} (V) \otimes_\kk \bbs^{\sigma (\beta, \gamma)} (V) = \ker \left( \bbs^{\sigma (\alpha)} (V) \otimes_\kk \bbs^{\sigma (\beta)} (V) \otimes_\kk \bbs^{\sigma (\gamma)} (V) \to \bbs^{\sigma (\alpha)} (V) \otimes_\kk \bbs^{\sigma (\beta \odot \gamma)} (V) \right),$$
%$$ \bbs^{\sigma (\alpha,\beta)} (V) \otimes_\kk \bbs^{\sigma ( \gamma)} (V) = \ker \left( \bbs^{\sigma (\alpha)} (V) \otimes_\kk \bbs^{\sigma (\beta)} (V) \otimes_\kk \bbs^{\sigma (\gamma)} (V) \to \bbs^{\sigma (\alpha,\beta)} (V) \otimes_\kk \bbs^{\sigma (\gamma)} (V) \right).$$}
\end{proof}

%%%%%%%%%%%%%%%%%%%%%%%%%%%%
\section{The complex of ribbons and proof of Theorem~\ref{skew-Schur-functor-theorem}}
\label{resolution-section}
%%%%%%%%%%%%%%%%%%%%%%%%%%%%

The goal of this section is to produce an explicit $GL(V)$-equivariant minimal $R$-free resolution of $M$,
where $R=\veralg{d}$ and $M=\vermod{d}{ r}$, with form as predicted by Theorem~\ref{skew-Schur-functor-theorem}.

%%%%%%%
\subsection{The complex of ribbons}
\label{ribbon-complex-subsection}

The {\it complex of ribbons} was mentioned
in the Introduction.  It has a differential
induced from a simple tensor-degree-lowering map (which is {\it not} itself a differential) on a tensor algebra.
%\begin{definition}\label{tensor-algebra-def} \rm
%Let $A$ be a commutative ring with $1$, and $U$ a free $A$-module.  The {\it tensor algebra\footnote{Actually, we will only define here its $A$-module structure, and not its full algebra structure.}} on $U$ is the direct sum
%$$ T_A(U):=\bigoplus_{n=0}^\infty T^n_A(U) \quad \text{ where }T^n_A(U):=\underbrace{U \otimes_A U \otimes_A \cdots \otimes_A U}_{n \text{ factors}} $$
%\end{definition}
%In our setting, $A$ will always be a a graded, connected $\kk$-algebra, so that free $A$-modules $U$ are all of the form $U=A \otimes_\kk \bar{U}$ for some $\kk$-vector space $\bar{U}$; for example, let $\bar{U} := U / A_+ U$.  Therefore one can re-cast $T_A(U)$ as follows:
%$$T_A(U)=T_A(A \otimes_\kk \bar{U}) \cong A \otimes_\kk T_\kk(\bar{U}). $$
%Note that for $U \overset{\varphi}{\rightarrow} A$ any $A$-linear map, one has an $A$-linear extension defined by
%$$
%\begin{array}{rcl}
%T_A(U) &\overset{d_\varphi}{\longrightarrow} &T_A(U)\\ u_1 \otimes u_2 \otimes \cdots \otimes u_\ell & \longmapsto & \varphi(u_1) u_2 \otimes \cdots \otimes u_\ell \end{array}
%$$
%which gives a (tensor-)degree-lowering $A$-module endomorphism of $T_A(U)$. 
%We now wish to specialize to the situation where $$ A=S=\kk[x_1,\ldots,x_n]=S(V),$$ with $V=\kk^n$.  We also wish to take $\bar{U}=S$ itself, so that our free $S$-module is $U=S \otimes_\kk S$, along with the $S$-linear map $\varphi$ being multiplication:
%$$
%\begin{array}{rcl}
%U = S \otimes_\kk S &\overset{\varphi}{\longrightarrow & S\\ s_1 \otimes s_2 &\longmapsto& s_1 s_2 \end{array} $$
Specifically,
define a tensor-degree-lowering $S$-linear endomorphism $\partial$ on 
$$
S \otimes_\kk T_\kk(S):=\bigoplus_{\ell=0}^\infty
S \otimes_\kk T_\kk^{\ell}(S)
$$
via this formula: 
$$
\begin{array}{rcl}
S \otimes_\kk T_\kk(S) &\overset{\partial}{\longrightarrow}& S \otimes_\kk T_\kk(S) \\
s \otimes_\kk (s_1 \otimes_\kk s_2 \otimes_\kk s_3 \otimes_\kk \cdots \otimes_\kk s_\ell) &\longmapsto& 
s \cdot s_1 \otimes_\kk (s_2 \otimes_\kk s_3 \otimes_\kk\cdots \otimes_\kk s_\ell)
\end{array}
$$

%Recall that for each composition $\alpha=(\alpha_1,\ldots,\alpha_\ell)$, the ribbon skew diagram $\sigma(\alpha)$ has row sizes given by $\alpha_1,\ldots,\alpha_\ell$ from bottom-to-top, with one column of overlap between consecutive rows. 
The definition in Section~\ref{general-linear-review-section}
of $\schurfunctor^{\sigma(\alpha)}$,
gives an inclusion of
(free) $\kk$-modules
$$
\schurfunctor^{\sigma(\alpha)}
\subseteq 
S^{\alpha_1} \otimes \cdots \otimes S^{\alpha_\ell}
\subseteq 
T_\kk^\ell(S)
$$
and hence also $S$-module inclusions
$$
S \otimes_\kk \schurfunctor^{\sigma(\alpha)}
\quad \subseteq \quad 
S \otimes_\kk \left( S^{\alpha_1} \otimes \cdots \otimes S^{\alpha_\ell} \right)
\quad \subseteq \quad 
S \otimes T_\kk^\ell(S) 
$$

\begin{definition} \rm
Define for each $\ell \geq 1$ an $S$-submodule $\fancyR_{\ell}$ of $S \otimes T_\kk^\ell(S)$ by
$$
\fancyR_{\ell}
:=S\otimes_\kk \left( \bigoplus_{\alpha} \schurfunctor^{\sigma(\alpha)} \right),
$$
where the direct sum is over compositions $\alpha=(\alpha_1,\ldots,\alpha_\ell)$ of length $\ell$. Compile them as
$$
\fancyR_{\bullet} := \bigoplus_{\ell =1}^\infty \fancyR_\ell.
$$
\end{definition}

\begin{prop}\label{prop:complexofRibbons}
The $S$-module endomorphism
$\partial$ on $S \otimes T_\kk(S)$
has these properties:
\begin{itemize}
\item[(i)] It restricts to an $S$-module endomorphism of $\fancyR_{\bullet}$.
\item[(ii)] Upon restriction to $\fancyR_{\bullet}$, it
satisfies $\partial^2=0$.
\end{itemize}
\end{prop}
For this reason, we will call $(\fancyR_\bullet, \partial)$
the {\it complex of ribbons}.

\begin{proof}
To prove (i),
for
$\alpha=(\alpha_1,\alpha_2,\ldots,\alpha_\ell)$,
let $\hat{\alpha}:=(\alpha_2,\ldots,\alpha_\ell)$, and
note $\partial$ gives a map
$$
\bbs^{\sigma(\alpha)}
\cong
S^0 \otimes_\kk \bbs^{\sigma(\alpha)}
\overset{\partial}{\longrightarrow} 
S^{\alpha_1} \otimes_\kk \bbs^{\sigma(\hat{\alpha})}
$$
which is the same as the
map $\Delta$ from a special case of Proposition~\ref{concatenation-near-concatenation-prop}.
Consequently, by $S$-linearity,
$\partial$ maps
$S \otimes_\kk \bbs^{\sigma(\alpha)}
\to
S \otimes_\kk \bbs^{\sigma(\hat{\alpha})},
$
so that it restricts to $\fancyR_\bullet$.

To prove (ii), introduce $\hat{\hat{\alpha}}:=(\alpha_3,\alpha_4,\ldots,\alpha_\ell)$,
and note that, by $S$-linearity,
it suffices to check that this restricted composite $\partial^2=0$:
$$
\bbs^{\sigma(\alpha)}
\cong
S^0 \otimes_\kk \bbs^{\sigma(\alpha)}
\overset{\partial}{\longrightarrow} 
S^{\alpha_1} \otimes_\kk \bbs^{\sigma(\hat{\alpha})}
\overset{\partial}{\longrightarrow} 
S^{\alpha_1+\alpha_2} \otimes_\kk \bbs^{\sigma(\hat{\hat{\alpha}})}.
$$
We claim $\partial^2=0$ here, because one can check $\partial^2$ coincides with the composite of two maps
$\nabla^{\{1\}}_{\varnothing} \circ i$ inside a particular instance of the 
exact sequence \eqref{lift-of-HG-for-n=3} where
$\underline{\alpha}:=( (\alpha_1), (\alpha_2), \hat{\hat{\alpha}} )$:

$$
	\bbs^{\sigma (\alpha)} 
	\overset{i}{\longrightarrow} 
	S^{\alpha_1}  \otimes_\kk S^{\alpha_2}  \otimes_\kk \bbs^{\sigma(\hat{\hat{\alpha}})}	\xrightarrow[]{\left[\begin{matrix} \nabla^{\{1\}}_{\varnothing} \\ \nabla^{\{2\}}_{\varnothing} \end{matrix}\right]}
	\begin{matrix}
	S^{\alpha_1+\alpha_2} \otimes_\kk \bbs^{\sigma (\hat{\hat{\alpha}})} \\
	\oplus\\
	S^{\alpha_1} \otimes_\kk \bbs^{\sigma( (\alpha_2) \odot \hat{\hat{\alpha}})}
	\end{matrix}
$$
Exactness here implies vanishing of both composites
$\nabla^{\{1\}}_{\varnothing} \circ i
=0
=\nabla^{\{2\}}_{\varnothing} \circ i$,
so $\partial^2=0$.
\end{proof}

The next lemma analyzes some of the cycles and
boundaries within certain summands of the complex of ribbons $(\fancyR_\bullet,\partial)$, and forms a key step in many of our later results.  To state it, recall that for 
a composition $\alpha = (\alpha_1 , \alpha_2, \dots , \alpha_\ell)$,
letting $\hat{\alpha}:=(\alpha_2,\ldots,\alpha_\ell)$ as usual, one can restrict $\partial$ in the complex of ribbons $\fancyR_\bullet$ to a map
$$
\partial^\alpha : 
S \otimes_\kk \bbs^{\sigma(\alpha)}  \to S \otimes_\kk \bbs^{\sigma (\hat{\alpha})}.
$$
The map $\partial$ and this restriction $\partial^\alpha$ are homogeneous with respect to the usual grading on $\fancyR$ inherited from the
grading on $S \otimes_\kk T_\kk(S)$, where
$S^p \otimes_\kk \bbs^{\sigma(\alpha)}$
lies within the homogenenous component of degree
$p+|\alpha|$. The corresponding homogeneous component of the map $\partial^\alpha$ is therefore a map
$$
(\partial^\alpha)_{p+|\alpha|}: 
S^p \otimes_\kk \bbs^{\sigma(\alpha)}
\rightarrow 
S^{p+\alpha_1} \otimes_\kk \bbs^{\sigma(\hat{\alpha})}.
$$

\begin{lemma}\label{lem:ribbonCxLemma}
Let $p>0$ be any positive integer.
\begin{enumerate}
\item[(i)]
One has a $GL(V)$-isomorphism
$$
\ker (\partial^\alpha)_{p + |\alpha|} \cong \bbs^{\sigma (p,\alpha)} .
$$
\item[(ii)]
Moreover, one has equalities
$$
\ker (\partial^\alpha)_{p + |\alpha|} 
= \im (\partial^{(q,\alpha)})_{p  + |\alpha|}  \quad \textrm{for all} \ 0<q \leq p.$$ 
\end{enumerate}
\end{lemma}

\begin{proof}
To prove (i), start by 
factoring $(\partial^\alpha)_{p+|\alpha|}=\pi \circ i$ as the following composite
%{\tiny % https://q.uiver.app/?q=WzAsNCxbMCwwLCJTXntcXGVsbH0gKFYpIFxcb3RpbWVzX1xca2sgXFxiYnNee1xcc2lnbWEgKFxcYWxwaGFfMSAsIFxcZG90cyAsIFxcYWxwaGFfbil9IChWKSJdLFsyLDAsIlNee1xcZWxsICsgXFxhbHBoYV8xfSAoVikgXFxvdGltZXNfXFxrayBcXGJic157XFxzaWdtYSAoXFxhbHBoYV8yICwgXFxkb3RzICwgXFxhbHBoYV9uKX0gKFYpIl0sWzEsMSwiU157XFxlbGx9IChWKSBcXG90aW1lc19cXGtrIFNee1xcYWxwaGFfMX0gKFYpIFxcb3RpbWVzX1xca2sgXFxiYnNee1xcc2lnbWEgKFxcYWxwaGFfMiAsIFxcZG90cyAsIFxcYWxwaGFfbil9IChWKSJdLFsxLDIsIlxca2VyIFxccGkgPSBcXGJic157XFxzaWdtYSAoXFxlbGwgLCBcXGFscGhhXzEpfSAoVikgXFxvdGltZXNfXFxrayBcXGJic157XFxzaWdtYSAoXFxhbHBoYV8yICwgXFxkb3RzICwgXFxhbHBoYV9uKX0gKFYpIl0sWzAsMiwiXFxpb3RhIiwwLHsic3R5bGUiOnsidGFpbCI6eyJuYW1lIjoiaG9vayIsInNpZGUiOiJ0b3AifX19XSxbMCwxXSxbMiwxLCJcXHBpIiwwLHsic3R5bGUiOnsiaGVhZCI6eyJuYW1lIjoiZXBpIn19fV0sWzMsMiwiIiwwLHsic3R5bGUiOnsidGFpbCI6eyJuYW1lIjoiaG9vayIsInNpZGUiOiJ0b3AifX19XV0=
\[
S^p \otimes_\kk \bbs^{\sigma (\alpha)}
\quad
\overset{\iota}{\hookrightarrow}
\quad
S^p \otimes_\kk S^{\alpha_1} \otimes_\kk \bbs^{\sigma(\hat{\alpha})} 
\quad
\overset{\pi}{\twoheadrightarrow}
\quad
S^{p + \alpha_1} \otimes_\kk \bbs^{\sigma(\hat{\alpha})}
\]
where $\iota$ is injective and $\pi$ is surjective,
using Proposition~\ref{tensoring-with-skews-is-exact-prop}
together with these facts:
\begin{itemize}
\item $\iota=1 \otimes_\kk \Delta_{(\alpha_1),\hat{\alpha}}$ 
where $\Delta_{(\alpha_1),\hat{\alpha}}$ is an instance of the injection in Proposition~\ref{concatenation-near-concatenation-prop},
\item $\pi=m \otimes_\kk 1$ where
$m: S^p \otimes S^{\alpha_1} \rightarrow S^{p+\alpha_1}$ is multiplication in $S=S(V)$.
\end{itemize}
Therefore one has that 
$$
\begin{aligned}
\ker (\partial^\alpha)_{p + |\alpha|} 
&= \im (\iota) \cap \ker (\pi) \\
&=
\iota\left( S^p \otimes_\kk \bbs^{\sigma (\alpha)}  \right) \cap \left( \bbs^{\sigma (p , \alpha_1)}  \otimes_\kk \bbs^{\sigma (\hat{\alpha})}  \right)\\
&\cong \bbs^{\sigma (p,\alpha_1,\hat{\alpha})} \\ &=\bbs^{\sigma (p ,\alpha)} ,
\end{aligned}
$$
where the identification 
$\ker (\pi) = \bbs^{\sigma (p , \alpha_1)}  \otimes_\kk \bbs^{\sigma (\hat{\alpha})} $ uses injectivity
of the map $\Delta_{(p,\alpha_1),\hat{\alpha}}$
in Proposition~\ref{concatenation-near-concatenation-prop}
(along with Proposition~\ref{tensoring-with-skews-is-exact-prop}),
and the isomorphism of the intersection with
$\bbs^{\sigma (p,\alpha_1,\hat{\alpha})} $
applies Proposition \ref{prop:intersectingSchurs}
to the three compositions $(p), (\alpha_1), \hat{\alpha}$.

Assertion (ii) will then follow from assertion (i), 
together with the commutativity of this diagram for $0 < q \leq p$:
% https://q.uiver.app/?q=WzAsNCxbMCwwLCJTXntcXGVsbCAtIG19IChWKSBcXG90aW1lc19cXGtrIFxcYmJzXntcXHNpZ21hIChtLFxcYWxwaGFfMSAsIFxcZG90cyAsIFxcYWxwaGFfbil9IChWKSJdLFsyLDAsIlNeXFxlbGwgKFYpIFxcb3RpbWVzX1xca2sgXFxiYnNee1xcc2lnbWEgKFxcYWxwaGFfMSAsIFxcZG90cyAsIFxcYWxwaGFfbil9IChWKSJdLFswLDEsIlxcYmJzXntcXHNpZ21hKFxcZWxsICwgXFxhbHBoYV8xICwgXFxkb3RzICwgXFxhbHBoYV9uKX0gKFYpIl0sWzIsMSwiU15cXGVsbCAoVikgXFxvdGltZXNfXFxrayBTXntcXGFscGhhXzF9IChWKSBcXG90aW1lc19cXGtrIFxcYmJzXntcXHNpZ21hKCBcXGFscGhhXzIgLCBcXGRvdHMgLCBcXGFscGhhX24pfSAoVikiXSxbMCwxLCJcXHBhcnRpYWxeeyhtLFxcYWxwaGEpfV97XFxlbGwgLSBtICsgfFxcYWxwaGF8fSJdLFswLDIsIiIsMix7InN0eWxlIjp7ImhlYWQiOnsibmFtZSI6ImVwaSJ9fX1dLFsyLDMsIiIsMix7InN0eWxlIjp7InRhaWwiOnsibmFtZSI6Imhvb2siLCJzaWRlIjoidG9wIn19fV0sWzEsMywiIiwwLHsic3R5bGUiOnsidGFpbCI6eyJuYW1lIjoiaG9vayIsInNpZGUiOiJ0b3AifX19XV0=
\[\begin{tikzcd}
	{S^{p - q}  \otimes_\kk \bbs^{\sigma (q,\alpha)} } && {S^p  \otimes_\kk \bbs^{\sigma (\alpha)} } \\
	{\bbs^{\sigma(p , \alpha)} } && {S^p  \otimes_\kk S^{\alpha_1}  \otimes_\kk \bbs^{\sigma( \hat{\alpha})} }
	\arrow["{(\partial^{(q,\alpha)})_{p+ |\alpha|}}", from=1-1, to=1-3]
	\arrow[two heads, from=1-1, to=2-1]
	\arrow[hook, from=2-1, to=2-3]
	\arrow[hook, from=1-3, to=2-3]
\end{tikzcd}\]
Surjectivity of the left vertical map follows because it is the surjection $m_{(p-q), (q,\alpha)}$
in Proposition~\ref{concatenation-near-concatenation-prop},
noting that the compositions $(p-q)$ and $(q,\alpha)$ have
$(p-q) \odot (q,\alpha)=(p,\alpha)$.
\end{proof}

%%%%%%%
\subsection{Resolution of the Veronese modules}
\label{resolution-subsection}

In this section, fix $d,r \geq 1$, and the Veronese ring $R=\veralg{d}$ with the $R$-module $M=\vermod{d}{ r}$. We will show that a minimal $R$-free resolution of $M$ is obtained by restricting to a uniform class of ribbons in the ribbon complex $\fancyR_\bullet$. 

Since $M$ is minimally generated as an $R$-module by the $r^{th}$ homogeneous component $S_r=S^r(V)=\schurfunctor^{\sigma(r)}$ of $S$, one can
start an $R$-free resolution of $M$ with the surjection
$$
\partial_0: 
R \otimes_\kk \schurfunctor^{\sigma(r)}=
R \otimes_\kk S_r
\twoheadrightarrow M
$$
given by multiplication within $S$.
Note that to make $\partial_0$ homogeneous, one should shift the degree of the $R$-basis elements for its source into degree $r$, making it
$R \otimes_\kk \schurfunctor^{\sigma(r)}(-r)$.
The following is then a precise version of Theorem~\ref{skew-Schur-functor-theorem} from the Introduction.

\begin{thm}
\label{explicit-resolution-theorem}
The map $\partial_0$ starts a $GL(V)$-equivariant minimal $R$-free resolution of $M$
$$
R \otimes_\kk \schurfunctor^{\sigma(r)}
\overset{\partial_1}{\leftarrow} R \otimes_\kk \schurfunctor^{\sigma(d,r)}
\overset{\partial_2}{\leftarrow} \cdots 
\leftarrow R \otimes_\kk \schurfunctor^{\sigma(d^{i-1},r)}
\overset{\partial_i}{\leftarrow} R \otimes_\kk \schurfunctor^{\sigma(d^i,r)}
\leftarrow \cdots
$$
obtained by  restricting the base ring in the complex of ribbons $(\fancyR_\bullet,d_\varphi)$ from $S$-modules to $R$-modules, and then
taking only the summands indexed by compositions of the form $\alpha=(d^i,r)$ 
for $i=0,1,2,\ldots$.
Here the free $R$-module in homological degree $i$ has its basis elements in
degree $di+r$, so it is
actually 
$R \otimes_\kk \schurfunctor^{\sigma(d^i,r)}(-(di+r))$.
\end{thm}

\begin{proof}
Since $\partial^2=0$, it remains to show $\ker(\partial_i) = \im(\partial_{i+1})$.
Using the $\NN$-grading on $S$ inherited by $R$ and $M$, 
one need only check this equality on each homogeneous component, that is, 
$\ker(\partial_i)_{d(i+m)+r} = \im(\partial_{i+1})_{d(i+m)+r}$
for all $m \geq 0$. Notice that if $m = 0$, then $\ker(\partial_i)_{d(i+m)+r} = 0$, since $(\partial_i)_{di+r}$ is just the natural inclusion
$$\bbs^{\sigma (d^i , r)}  \to S^d  \otimes_\kk \bbs^{\sigma (d^{i-1} , r)} .$$
When $m > 0$, applying Lemma \ref{lem:ribbonCxLemma} with $\alpha=(d^i,r)$, $p=dm$ and $q=d$ gives the equality:
$$
\begin{array}{ccccc}
\ker (\partial^{(d^i,r)})_{d(i+m)+r} 
&=&
\im (\partial^{(d,d^i,r)})_{d(i+m)+r}&=&\im (\partial^{(d^{i+1},r)})_{d(i+m)+r}\\
\Vert & & & &\Vert \\
\ker(\partial_i)_{d(i+m)+r} & & & & \im(\partial_{i+1})_{d(i+m)+r}.
\qedhere
\end{array}
$$
\end{proof}

\begin{example}
Take $d=3$ and $r=2$, 
so that $R=\veralg{3}$ and $M=\vermod{3}{2}$. Then the
$R$-free resolution of $M$ has the form
$$
R \otimes_\kk \schurfunctor^{
\ytableausetup{boxsize=0.3em}
\begin{ytableau}
\, & 
\end{ytableau}}
\leftarrow R \otimes_\kk \schurfunctor^{
\ytableausetup{boxsize=0.3em}
\begin{ytableau} 
 \none &\none & & \\ 
  & &  
\end{ytableau}
}
\leftarrow R \otimes_\kk \schurfunctor^{
\ytableausetup{boxsize=0.3em}
\begin{ytableau} 
\none& \none& \none &\none & &\\ 
 \none& \none&  & & \\  
   & &     
\end{ytableau}
}
\leftarrow \cdots.
$$
To demonstrate how the differential behaves, we choose an element from $R \otimes_\kk \schurfunctor^{
\ytableausetup{boxsize=0.3em}
\begin{ytableau} 
\none& \none& \none &\none & &\\ 
 \none& \none&  & & \\  
   & &     
\end{ytableau}
}$ and show that the differential composes to zero:
\begin{align*}
\ytableausetup{boxsize=0.9em}
    1 \otimes \left[\yshort{\none\none\none\none 12, \none\none 345, 678} \right] &\overset{\partial_2}{\longmapsto} x_6 x_7 x_8 \otimes \left[\yshort{\none\none 12, 345}\right] - x_3 x_6 x_7 \otimes \left[\yshort{\none\none 12, 845} \right] \\
    &\overset{\partial_1}{\longmapsto} x_3 x_4 x_5 x_6 x_7 x_8 \otimes \left[\,\, \yshort{12}\,\,\right] - x_1 x_3 x_4 x_6 x_7 x_8 \otimes \left[\,\, \yshort{52}\,\,\right] \\ &- \left(\,\, x_3 x_4 x_5 x_6 x_7 x_8 \otimes \left[\,\,\yshort{12}\,\,\right] - x_1 x_3 x_4 x_6 x_7 x_8 \otimes \left[\,\,\yshort{52}\,\,\right]\,\, \right) \\
     &=0.
\end{align*}
\end{example}

With Theorem~\ref{skew-Schur-functor-theorem} proven,
the next few subsections remark on its
consequences.

\subsection{Tor and Koszulity}
\label{tor-koszulity-subsection}

As mentioned in the Introduction,
Theorem~\ref{skew-Schur-functor-theorem} has a corollary.

\vskip.1in
\noindent
{\bf Corollary}~\ref{main-tor-corollary}
{\it
In the setting of Theorem~\ref{skew-Schur-functor-theorem}, 
$\Tor_i^R (M,\kk)_j$ vanishes for $j\neq di+r$, and
$$
\Tor_i^R (M,\kk)_{di+r}
\cong \schurfunctor^{\sigma(d^i,r)},
$$
as a polynomial $GL(V)$-representation.
}
\vskip.1in

\noindent
In particular, the resolution of  Theorem~\ref{skew-Schur-functor-theorem}
is pure, in the sense that its $i^{th}$ resolvent is a free $R$-module whose basis elements lie in degree $di+r$.

This corollary has consequences for $R$ and $M$ as Koszul algebras and modules. We recall here the
definitions of these concepts, and a few of their
basic properties.  Let $A$ be an associative graded connected $\kk$-algebra, so 
$A=\bigoplus_{m=0}^\infty A_m$ with $A_0=\kk$ and $A_i A_j \subseteq A_{i+j}$.  Say that $A$ is {\it standard graded} if it is generated
as a $\kk$-algebra by $A_1$.  

\begin{definition} \rm
In the above setting, one calls $A$ a {\it Koszul algebra} if the quotient $\kk=A/A_+$ has a {\it linear free $A$-resolution}, that is, one of the form
$$
0 \leftarrow \kk 
\leftarrow A
\leftarrow A(-1)^{\beta_1}
\leftarrow  A(-2)^{\beta_2}
\leftarrow \cdots.
$$
Equivalently, $\Tor^A_i(\kk,\kk)_j=0$ unless $j=i$.

Given a graded $A$-module $M$, one says that $M$ is a {\it Koszul module} over $A$ if it is generated in degree $0$ and has a linear free $A$-resolution, that is, of the form
$$
0 \leftarrow M 
\leftarrow A^{\beta_0(M)}
\leftarrow A(-1)^{\beta_1(M)}
\leftarrow  A(-2)^{\beta_2(M)}
\leftarrow \cdots
$$
Equivalently, $\Tor^A_i(M,\kk)_j=0$ unless $j=i$.
\end{definition}

The following basic facts appear, for example, in Positselski and Polishchuk\footnote{Actually, in \cite{polishchuk2005quadratic} the authors assume $\kk$ is a field, although their methods do not require it.  For Koszulity definitions and results over more general rings $\kk$, see Beilinson, Ginzburg and Soergel \cite{beilinson1996koszul}.} \cite{polishchuk2005quadratic}.

\begin{prop}
\label{PolischukPositselski-prop} Let $A$ be a Koszul algebra.
\begin{enumerate}[(a)]
    \item \cite[Chap. 3, Prop. 2.2]{polishchuk2005quadratic} Veronese subalgebras of $A$ are also Koszul.
    \item \cite[Chap. 2, Prop. 1.1]{polishchuk2005quadratic} Let $M$ be a Koszul module over $A$.  
Then for every $q \geq 0$,
the $q^{th}$ truncated module $M^{[q]}$ defined by 
$$
M^{[q]}_n = \begin{cases}
M_{q+n} & \text{ for }n \geq 0,\\
0 &\text{ for }n <0,
\end{cases}
$$
is also a Koszul module over $A$.
\end{enumerate}

\end{prop}

Taking $r=d$, one sees that  Corollary~\ref{main-tor-corollary} gives an
alternate proof that $R=\veralg{d}$ is a Koszul algebra,
after rescaling the grading
so that its algebra generators lie in degree $1$; this was
originally proven by Barcanescu and Manolache \cite{BarcanescuManolache}. 

On the other hand,  taking $r$ arbitrary, they also prove that each module $M=\vermod{d}{ r}$ is a Koszul $R$-module, after shifting its grading
so that the $R$-generators of $M$ lie in degree $0$.
This Koszulity of $\vermod{d}{ r}$ as an $R$-module was proven for $0 \leq r \leq d-1$ by Aramova, Barcanescu and Herzog \cite{AramovaBarcanescuHerzog}.  As an alternative to
using Corollary~\ref{main-tor-corollary}, one could deduce Koszulity of $\vermod{d}{ r}$
for arbitrary $r \geq 1$ from the known
$0 \leq r \leq d-1$  case as follows.
Write $r=q \cdot d + \hat{r}$ where $0 \leq\hat{r} \leq d-1$, and then apply 
Proposition~\ref{PolischukPositselski-prop}
with $A=R=\veralg{d}$, and $M=\vermod{d}{\hat{r}}$.
This shows that the truncated module $M^{[q]}=\vermod{d}{ r}$ is Koszul.

%%%%%%%
\subsection{A symmetric function identity}
\label{symmetric-function-identity-section}

Exactness of the resolution in Theorem~\ref{skew-Schur-functor-theorem} implies
an identity of $GL(V)$-characters
$$
\ch(R) \cdot \sum_{i=0}^\infty (-1)^i \ch(\Tor_i^R(M,\kk)) 
= \ch(M).
$$
Since as symmetric functions, one has
$$
\begin{aligned}
\ch(R) &=1+h_d +h_{2d} + h_{3d}+ \cdots,\\
\ch(M) &=h_r + h_{d+r} +h_{2d+r}+ \cdots,\\
\ch(\Tor_i^R(M,\kk)) & = s_{\sigma(d^i,r)},
\end{aligned}
$$
this becomes the following symmetric function identity:
\begin{equation}
\label{symmetric-function-identity}
h_r + h_{d+r} +h_{2d+r}+ \cdots
=\left( 1+h_d +h_{2d} + h_{3d}+ \cdots \right)
\cdot 
\sum_{i=0}^\infty (-1)^i s_{\sigma(d^i,r)}.
\end{equation}
In fact, if one gives an independent proof of the symmetric function identity \eqref{symmetric-function-identity}, this leads in the case where $\kk$ is a field of characteristic zero, 
to an alternate proof of Corollary~\ref{main-tor-corollary}:
\begin{itemize}
    \item first appeal to the Koszulity result for $M$ proven as in Section~\ref{tor-koszulity-subsection}, then
    \item use this to deduce that \eqref{symmetric-function-identity}
determines each of the $GL(V)$-characters $\ch(\Tor_i^R(M,\kk))$, 
\item which then determine each $\Tor_i^R(M,\kk)$ uniquely as a $GL(V)$-representation. 
\end{itemize}
We therefore explain here briefly how \eqref{symmetric-function-identity} connects to
known symmetric function identities.
For each $m=0,1,2,\ldots$, extracting the homogeneous component of degree $md+r$ shows that it is equivalent
to the identity
$
h_{md+r} 
= \sum_{i=0}^m (-1)^i h_{d(m-i)} \cdot s_{\sigma(d^i,r)}.
$
This can be rewritten, isolating the last term $\sigma(d^m,r)$ in the sum, as
\begin{equation}
\label{desired-symm-function-identity}
\begin{aligned}
s_{\sigma(d^m,r)}
&=+h_{d} \cdot s_{\sigma(d^{m-1},r)} 
-h_{2d} \cdot s_{\sigma(d^{m-2},r)}
+h_{3d} \cdot s_{\sigma(d^{m-3},r)}-\cdots\\
&\qquad + (-1)^{m-2} h_{(m-1)d} \cdot s_{\sigma(d^1,r)} 
+ (-1)^{m-1} h_{md} \cdot s_{\sigma(r)}\\
&\qquad\qquad + (-1)^m h_{md+r} 
\end{aligned}
\end{equation}
%\begin{equation}
%\label{desired-symm-function-identity}
%\begin{aligned}
%s_{\sigma(d,r,m)}
%=&+h_{d} \cdot s_{\sigma(d^{m-1},r)} \\
%&-h_{2d} \cdot s_{\sigma(d^{m-2},r)}\\
%&+h_{3d} \cdot s_{\sigma(d^{m-3},r)}\\
%&\vdots\\
%&+ (-1)^{m-2} h_{(m-1)d} \cdot s_{\sigma(d^1,r)} \\
%&+ (-1)^{m-1} h_{md} \cdot s_{\sigma(r)} \\
%&+ (-1)^m h_{md+r} 
%\end{aligned}
%\end{equation}
We claim that this identity \eqref{desired-symm-function-identity} is a consequence of the {\it Jacobi-Trudi formula} \cite[(5.4)]{Macdonald} \cite[\S 7.16]{Stanley-EC2} expressing $s_{\sigma(d^m,r)}$ as a determinant in $h_n$'s.  For a general skew shape $\lambda/\mu$ one has
$
s_{\lambda/\mu}
=\det[ 
h_{\lambda_i - \mu_j - i + j}
]_{i,j=1,2,\ldots,\ell(\lambda)}.
$
For the special case of the ribbon shape $\sigma(d^m,r)$, this takes the following form:
$$
s_{\sigma(d^m,r)}
=\det
\left[ 
\begin{matrix}
h_r & h_{d+r} & h_{2d+r}  & h_{3d+r}& \cdots &  h_{md+r} \\
1   & h_d & h_{2d} & h_{3d} & \cdots &h_{md} \\
0   & 1   & h_d & h_{2d} & \cdots &h_{(m-1)d} \\
0   & 0   & 1   & h_d &\cdots& h_{(m-2)d}\\
\vdots   &\vdots  & &  &\ddots &\vdots \\
0  &0 & \cdots   & 0 & 1   & h_d 
\end{matrix}
\right]
$$
The identity
\eqref{desired-symm-function-identity}
is simply the above determinant expanded along its last column.
%See Example~\ref{Jacobi-Trudi-example} below.

%\begin{example}
%\label{Jacobi-Trudi-example}
%The Jacobi-Trudi formula says 
%$$
%\begin{aligned}
%&s_{\sigma(4^3,7)}
%= \quad 
%s_{\ytableausetup{boxsize=0.3em}
%\begin{ytableau} 
%\none& \none& \none & \none& \none& \none &\none &\none&\none & &  & & & & &\\  \none& \none& \none &  \none& \none& \none & & &  & \\  \none& \none& \none &   & & &   \\& & &   \\ 
%\end{ytableau}}
%=\quad \det
%\left[ 
%\begin{matrix}
%h_7 & h_{11} & h_{15} & h_{19} \\
%1   & h_4 & h_{8} & h_{12} \\
%0   & 1   & h_4 & h_{8}  \\
%0   & 0   & 1   & h_4 
%\end{matrix}
%\right] \\
%& \\
%&=+h_4 \cdot \det
%\left[ 
%\begin{matrix}
%h_7 & h_{11} & h_{15} \\
%1   & h_4 & h_{8} \\
%0   & 1   & h_4  
%\end{matrix}
%\right]
%- h_8 \cdot \det
%\left[ 
%\begin{matrix}
%h_7 & h_{11} & h_{15}  \\
%1   & h_4 & h_{8}   \\
%0   & 0   & 1 
%\end{matrix}
%\right] \\
%&\quad +h_{12} \cdot \det
%\left[ 
%\begin{matrix}
%h_7 & h_{11} & h_{15} \\
%0   & 1   & h_4  \\
%0   & 0   & 1   
%\end{matrix}
%\right]
%- h_{19} \cdot \det
%\left[ 
%\begin{matrix}
%1   & h_4 & h_{8} \\
%0   & 1   & h_4   \\
%0   & 0   & 1   
%\end{matrix}
%\right] \\
%& \\
%&=+h_4 \cdot \det
%\left[ 
%\begin{matrix}
%h_7 & h_{11} & h_{15} \\
%1   & h_4 & h_{8} \\
%0   & 1   & h_4  
%\end{matrix}
%\right]
%- h_8 \cdot \det
%\left[ 
%\begin{matrix}
%h_7 & h_{11}  \\
%1   & h_4    
%\end{matrix}
%\right]
%+h_{12} \cdot \det
%\left[ 
%\begin{matrix}
%h_7
%\end{matrix}
%\right]
%- h_{19} \\
%& \\
%&=+h_4 \cdot s_{\sigma(4^2,7)}
%- h_8 \cdot s_{\sigma(4^1,7)}
%+h_{12} \cdot s_{\sigma(7)}
%- h_{15}.
%\end{aligned}
%$$
%\end{example}

\subsection{A curious symmetry}
\label{curious-omega-symmetry}
When $r=1$ and $d=2$, then $R=\veralg{2}$ and $M=\vermod{2}{1}$ are the polynomials of even
degree and odd degree in $S=\kk[x_1,\ldots,x_n]$, respectively.  In this case, Theorem~\ref{skew-Schur-functor-theorem}
shows that $M$ has its minimal $R$-free resolution of the form
$$
0 \leftarrow M 
\leftarrow
R \otimes \schurfunctor^{
\ytableausetup{boxsize=0.3em}
\begin{ytableau}
\, 
\end{ytableau}}
\leftarrow R \otimes \schurfunctor^{
\ytableausetup{boxsize=0.3em}
\begin{ytableau} 
 \, & \,\\ 
\, \\ 
\end{ytableau}
}
\leftarrow R \otimes \schurfunctor^{
\ytableausetup{boxsize=0.3em}
\begin{ytableau} 
  \none&\, &\,\\
  & \\ 
  \\
\end{ytableau}
}
\leftarrow R \otimes \schurfunctor^{
\ytableausetup{boxsize=0.3em}
\begin{ytableau} 
  \none&\none&    &  \\
  \none&     &\\ 
  &   \\
  \\
\end{ytableau}
}
\leftarrow \cdots
$$
involving only ribbon skew Schur functors indexed by skew shapes that
are {\it invariant under transposition}, that is,
their associated skew Schur functions are
stable under the {\it fundamental
involution} $\Lambda \overset{\omega}{\longrightarrow }\Lambda$ swapping $h_n \leftrightarrow e_n$ and $s_\lambda \leftrightarrow s_{\lambda^t}$; see Stanley \cite[\S 7.6, 7.15]{Stanley-EC2}.  In this case, one could rewrite \eqref{symmetric-function-identity}
as follows:
$$
\sum_{i=0}^\infty (-1)^i \ch( \Tor_i^{\veralg{2}} ( \vermod{2}{1},\kk) ) 
=\frac{h_1 + h_{3} +h_{5}+ \cdots}{1+h_2 +h_{4} + h_{6}+ \cdots}
=s_{\ytableausetup{boxsize=0.3em}
\begin{ytableau} 
\, \\ 
\end{ytableau}} 
- s_{\ytableausetup{boxsize=0.3em}
\begin{ytableau} 
 \, & \,   \\  
  \,    \\ 
\end{ytableau}} 
+ s_{\ytableausetup{boxsize=0.3em}
\begin{ytableau} 
 \none &  &  \\  
 &   \\  
   \\ 
\end{ytableau}}
- s_{\ytableausetup{boxsize=0.3em}
\begin{ytableau} 
\none& \none &  &  \\  
\none&  &   \\  
 &   \\
 \\
\end{ytableau}}
+ \cdots $$
There is an {\it a priori} reason involving symmetric functions for why this transpose-invariance occurs, as follows.  Recall that
$$
\begin{aligned}
H(t)&:=1+h_1 t+h_2t^2+\cdots\\
E(t)&:=1+e_1 t+e_2t^2+\cdots
\end{aligned}
$$
satisfy $H(t)E(-t)=1$.  Also $\omega$ swaps $H(t) \leftrightarrow E(t)$
since it swaps $h_n \leftrightarrow e_n$.
Now note 
$$
\frac{h_1 + h_{3} +h_{5}+ \cdots}{1+h_2 +h_{4} + h_{6}+ \cdots}
 = \frac{ \frac{1}{2}(H(1) - H(-1)) }
         {\frac{1}{2}(H(1) + H(-1)) }
  = \frac{ 1 - H(-1)/H(1) }                  { 1 + H(-1)/H(1) } 
 = \frac{ 1 - H(-1)\cdot E(-1) }                  { 1 + H(-1)\cdot E(-1) } 
$$
which is stable under $\omega$.

\begin{question}
Is there a more conceptual explanation for this $\omega$-stability of
$\Tor^{\veralg{2}}(\vermod{2}{1},\kk)$?
\end{question}

%%%%%%%
\subsection{Poset homology}
\label{poset-homology-section}
The rings $S=\kk[x_1,\ldots,x_n], R=\veralg{d}$ and $R$-modules $M=\vermod{d}{ r}$ are examples of a known set-up involving affine semigroups, where one can re-interpret $\Tor_i^R(M,\kk)$ in terms of homology of certain partially ordered sets ({\it posets}).  We recall this set-up here, which in general has three players:
\begin{itemize}
    \item An ambient affine semigroup ring $S$ over the scalar ring $\kk$,
    with $\kk$-basis the semigroup elements $\{ \sigma \}$ (this applies to
    $S=\kk[x_1,\ldots,x_n]$ whose semigroup is $\NN^n$).
    \item A $\kk$-subalgebra $R$ of $S$ which is the affine semigroup ring for a subsemigroup with $\kk$-basis elements the semigroup elements $\{ \rho \}$ (this applies to $R=\veralg{d}$).
    \item An $R$-submodule $M$ of $S$ coming from a semigroup $R$-submodule, with $\kk$-basis elements the semigroup elements $\{ \mu \}$ (this applies to $M=\vermod{d}{ r}$).
\end{itemize}
In this context, define a partial order $\leq_R$ on $M$ by $\mu_1 <_R \mu_2$ if $\mu_1 \rho = \mu_2$ for some $\rho$ in $R$.
Let $M_{< \mu}$ be the subposet consisting of the elements strictly below $\mu$ in this order.  Also recall that the {\it order complex}
$\Delta P$ of a poset $P$ is the abstract simplicial complex having
$P$ as vertex set and a simplex for each {\it chain} (=totally ordered subset) of $P$.

One then has the following result, whose idea goes back at least as far as Laudal and Sletsj{\o}e \cite{LaudalSletsjoe}; see also Herzog, Reiner and Welker \cite{HerzogRWelker}, Peeva, Reiner and Sturmfels \cite{PeevaRSturmfels}.  As in those sources, the proof (whose details we omit) comes from computing $\Tor^R_i(M,\kk)_\mu$ starting with the augmented bar resolution of $\kk$, next applying $M \otimes_R(-)$, then extracting the $\mu$-graded component, and finally taking homology.

\begin{prop}
\label{Tor-as-semigroup-poset-homology}
In the above setting, for each semigroup element $\mu$ in $M$ one has
$$
\Tor^R_i(M,\kk)_\mu \cong 
\tilde{H}_{i-1}(\Delta(M_{<\mu}),\kk),
$$
equivariant for any group $G$ acting on $S$ that stabilizes $R, M$ and the multidegree $\mu$.
\end{prop}

This immediately gives the following corollary of Theorem~\ref{skew-Schur-functor-theorem}.

\begin{cor}
\label{poset-homology-corollary}
For any ring $\kk$, fix $d,r \geq 1$, and let $R=\veralg{d},M=\vermod{d}{ r}$
as usual. Then for any multidegree
$\mathbf{a}=(a_1,\ldots,a_n)$ in $\NN^n$
with $|\mathbf{a}|:=\sum_j a_j \geq r$
and $|\mathbf{a}| \equiv r \bmod{d}$, one has
$\tilde{H}_{i-1}(\Delta(M_{<\mu}),\kk)=0$
unless $|\mathbf{a}|=di+r$, in which case
$$
\tilde{H}_{i-1}(\Delta(M_{<\mu}),\kk)
\cong 
( 
\schurfunctor^{\sigma(d^i,r)}
)_{\xx^{\mathbf{a}}}.
$$
Here $( 
\schurfunctor^{D}
)_{\xx^{\mathbf{a}}}$
denotes the $\xx^{\mathbf{a}}$-weight space in the polynomial $GL(V)$-representation
$\schurfunctor^{D}$, and
the isomorphism is equivariant with respect to the subgroup $\symm_{\mathbf{a}}$ within $\symm_n$ stabilizing $\mathbf{a}$.
\end{cor}

An important special case occurs by specializing to $\mathbf{a}=(1,1,\ldots,1)=(1^{di+r})$,
so that $\mu=\xx^{\mathbf{a}}=x_1 x_2 \cdots x_{di+r}$, and both the left
and right sides of Corollary~\ref{poset-homology-corollary} have simpler interpretations.
%Schur-Weyl duality as in Section~\ref{Schur-Weyl-duality-section}, lets one rephrase Theorem~\ref{skew-Schur-functor-theorem} as follows.
%\vskip.1in
%\noindent
%{\bf Theorem~\ref{skew-Schur-functor-theorem}${}^\prime$.}
%{\it
%Fix $d,r \geq 1$, with $R=\veralg{d}, M=\vermod{d}{ r}$. Then for $i\geq 0$, one has $\Tor_i^R (M,\kk)_{x_1 x_2 \cdots x_m}=0$ unless $m=id+r$, where one has a $\kk\symm_m$-module isomorphism 
%$$
%\Tor_i^R (M,\kk )_{x_1 x_2 \cdots x_m} \cong \specht^{\sigma(d^i,r)}.
%$$
%}
%Proposition~\ref{Tor-as-semigroup-poset-homology} gives us this third equivalent rephrasing of Theorems~\ref{skew-Schur-functor-theorem} and ~\ref{skew-Schur-functor-theorem}${}^\prime$.
%\vskip.1in
%\noindent
%{\bf Theorem~\ref{skew-Schur-functor-theorem}${}^{\prime\prime}$.}
%{\it
%Fix $d,r \geq 1$, with $R=\veralg{d}, M=\vermod{d}{ r}$. Then for $i \geq 0$, the elements $\mu=x_1 x_2 \cdots x_{di+r}$ in $M$ have $\tilde{H}_{j}( \Delta(M_{< \mu}),\kk)=0$ unless $j=i-1$, where
%$$
%\tilde{H}_{i-1}( \Delta(M_{< \mu},\kk)
%\cong \specht^{\sigma(d^i,r)}.
%$$
%}
On the right side, letting $m=di+r$,
the $x_1x_2\cdots x_m$-weight space in the $GL(V)$-representation
$\schurfunctor^{\sigma(d^i,r)}$ is the
skew Specht module $\specht^{\sigma(d^i,r)}$ for the symmetric group $\symm_m$, as discussed at the end of Section~\ref{Schur-Weyl-duality-section}.

On the left side, $M_{< \mu}$ is isomorphic to a well-studied poset from the literature on poset topology.  The supports of the squarefree monomials appearing in $M_{< \mu}$ are exactly the
proper subsets $A \subsetneq \{1,2,\ldots,id+r\}$ with cardinality
$|A| \geq r$ and $|A| \equiv r \bmod{d}$.
This poset $M_{< \mu}$ is therefore obtained
from the {\it Boolean algebra} $2^{[m]}$ of all
subsets of $[m]:=\{1,2,\ldots,m\}$ where $m=id+r$,
by selecting the elements whose ranks lie
in $\{r,r+d,r+2d,\ldots, r+(i-1)d\}$.
Thus one can rephrase this special case 
of Corollary~\ref{poset-homology-corollary}
as asserting that the following:

\begin{cor}
\label{squarefree-weight-poset-special-case-cor}
The order complex $\Delta$
for the rank-selected subposet of $2^{[m]}$ with $m=di+r$
allowing only the subsets whose ranks lie in $\{r,r+d,r+2d,\ldots, r+(i-1)d\}$
has $\tilde{H}_j(\Delta,\kk)=0$
unless $j=i-1$, in which case $\tilde{H}_{i-1}(\Delta,\kk) \cong \specht^{\sigma(d^i,r)}$ as an $\symm_m$-representation.
\end{cor}

Corollary~\ref{squarefree-weight-poset-special-case-cor} is
also an instance of a result\footnote{One also needs the $S_n$-representation isomorphism
$
\specht^{\sigma(r,d^i)} \cong
\specht^{\sigma(d^i,r)},
$
an instance of the more general isomorphism
$
\specht^{\sigma(\alpha)} \cong \specht^{\sigma(\rev(\alpha))}
$
where
$\rev(\alpha)=(\alpha_\ell,\alpha_{\ell-1},\dots,\alpha_2,\alpha_1)$.
The latter
follows, e.g., from Theorem~\ref{Solomon-ribbon-result}
using the poset anti-automorphism $A \mapsto \{1,2,\ldots,m\} \setminus A$ to the Boolean algebra $2^{[m]}$.} of
Solomon \cite[\S 6]{Solomon}, calculating the homology of all rank-selected subposets of a Boolean algebra $2^{[m]}$, which has  been re-examined many times; see Wachs \cite[\S3.4]{Wachs}. Solomon's result says the following \cite[Thm. 3.4.4]{Wachs}.

\begin{thm}
\label{Solomon-ribbon-result}
For any composition $\alpha=(\alpha_1,\ldots,\alpha_\ell)$ of $m$,
the order complex $\Delta$ of the
rank-selected subposet
of $2^{[m]}$ allowing only the subsets whose ranks lie in the partial sums
$$
\{\alpha_1, \,\, 
\alpha_1+\alpha_2, \,\, 
\alpha_1+\alpha_2+\alpha_3,\,\,
\ldots,\,\,
\alpha_1+\cdots+\alpha_{\ell-1}\}
$$
will have $\tilde{H}_i(\Delta;\kk)=0$ unless $i=\ell-2$, while
$
\tilde{H}_{\ell-2}(\Delta;\kk)
\cong 
\specht^{\sigma(\alpha)}
$
as an $\symm_m$-representation.
\end{thm}
 
\begin{remark} 
On the other hand, as discussed in Section~\ref{Schur-Weyl-duality-section}, when $\kk$ is a field of characteristic zero, one can recover a polynomial $GL(V)$-representation  such as $\Tor_i^R(M,\kk)$ for $R=\veralg{d}, M=\vermod{d}{r}$ from the $\symm_m$-representation on its $x_1 x_2 \cdots x_m$-weight space, where $m=\dim(V)$.  In this way, using Solomon's Theorem~\ref{Solomon-ribbon-result} to deduce Corollary~\ref{squarefree-weight-poset-special-case-cor}
gives yet another alternate proof of Corollary~\ref{main-tor-corollary} over fields of characteristic zero.
\end{remark}

%%%%%%%%%%%%%%%%%%%%%%%%%%%%
\section{On tensor products, $\Tor$ and $\Hom$ between the modules}
\label{tor-and-ext-between-modules-section}
%%%%%%%%%%%%%%%%%%%%%%%%%%%%

Our next goal is to compute $\Tor_i^R(M,M')$ and $\Hom_R(M,M')$ where
$$
\begin{aligned}
R &= \veralg{d},\\
M &= \vermod{d}{ r},\\
M' &= \vermod{d}{ r'}.
\end{aligned}
$$

We first dispense with the easy case where $\rank_\kk V =1$, that is, $S=\kk[x]$. This case is something of an outlier,
%because it is the only case for which the action of $\bbz / d \bbz \subset \gl (V)$ is not a small group action (see Definition \ref{def:setupForHoms}). 
but easily analyzed, since all of the modules
over $R=\veralg{d}=\kk[x^d]$
are free of rank $1$; that is, $\vermod{d}{ r} =x^r \kk[x^d]$. This immediately implies the following: 

\begin{prop}\label{prop: one-variable-ext-tor}
For $n=1$, the higher derived functors vanish, i.e.,
$$
\Tor^R_i(M,M')
=0
=\ext^i_R(M,M') \quad \text{ for }i \geq 1.
$$
and in the $i=0$ case, one has that both are free $R$-modules of rank one:
$$
\begin{array}{rcccl}
\Tor_0^R(M,M')
&=&x^r \kk[x^d] \otimes_{\kk[x^d]} x^{r'} \kk[x^d]&
\cong&
x^{r+r'} \kk[x^d],\\
\ext^0_R(M,M')
&=&
\hom_{\kk[x^d]} (x^{r}  \kk[x^d] , x^{r'} \kk[x^d] )
&\cong&
x^{r'-r} \kk[x^d] \quad \left( \subset \kk[x,x^{-1}] \right).
\end{array}
$$
\end{prop}

% The isomorphism on the second-to-last line
% comes from multiplying the two tensor factors.
% %$$x^r f(x^d) \otimes x^{r'}g(x^d)\longmapsto x^{r+r'}f(x^d)g(x^d)$$
% The (backwards) isomorphism
% on the last line
% sends an element of $x^{r'-r}\kk[x^d]$ to the  homomorphism $x^r\kk[x^d] \rightarrow x^{r'}\kk[x^d]$
% which multiplies by that element.

\subsection{Tensor products and proof of Theorem~\ref{thm: tensor-product} }
\label{tensor-section}

Recall the statement of the theorem.

\vskip.1in
\noindent
{\bf Theorem ~\ref{thm: tensor-product} .}
{\it
Fix $d,r,r' \geq 1$ and let $R=\veralg{d}$ with the three $R$-modules 
$$
\begin{aligned}
M&=\vermod{d}{ r},\\
M'&=\vermod{d}{ r'}, \\
M''&=\vermod{d}{ r''}, \quad \text{ where }r''=r+r'.
\end{aligned}
$$
\begin{itemize}
\item[(i)]
The multiplication map $M \otimes_R M' \overset{\varphi}{\rightarrow} M''$ gives rise to a $GL(V)$-equivariant short exact sequence of $R$-modules 
$$
0 \to \bbs^{\sigma(r,r')} (-r'') 
\to 
M \otimes_R M' 
\to M'' \to 0
$$
with the $R$-module 
$\bbs^{\sigma(r,r')} (-r'')$ concentrated in degree $r''$,
annihilated by $R_+$.
\item[(ii)] The sequence splits as $R$-modules, giving an $R$-module
isomorphism
$$
M \otimes_R M' \cong M''
\oplus \schurfunctor^{\sigma(r,r')}(-r'').
$$
\item[(iii)]
When $\binom{r+r'}{r}$ lies in $\kk^\times$,
the sequence also splits as
$\gl (V)$-representations.
\end{itemize}
}
\vskip.1in

%Consider the symmetric algebra $S=S(V)$, where $V$ is an $n$-dimensional $\kk$-vector space. Let $d\geq 1$ and write $R=\veralg{d}$.  Observe that the tensor product $\vermod{d}{ r}\otimes_{\veralg{d}} \vermod{d}{r'}$ naturally surjects onto the module $\vermod{d}{r+r'}$, and in the case when $\dim V = 1$ we know from Proposition \ref{prop: one-variable-ext-tor} that this is an isomorphism. More generally, it turns out this is a split surjection and its kernel is a ribbon Schur module with row lengths corresponding to $r$ and $r'$. The main goal of this section is to prove these statements.

\begin{remark}
Note that Theorem ~\ref{thm: tensor-product}
is consistent with the description of
$\Tor^R_0(M,M')$ in Proposition~\ref{prop: one-variable-ext-tor},
because $\schurfunctor^{\sigma(r,r')}(V)=0$
when $\dim(V)=1$ and $r,r' \geq 1$.
\end{remark}

\begin{remark}
Theorem ~\ref{thm: tensor-product}
holds for $r=0$ or $r'=0$ assuming
$\schurfunctor^{\sigma(0,r')}=\schurfunctor^{\sigma(r,0)}=0$.
\end{remark}

The essence of
Theorem~\ref{thm: tensor-product} is
the next lemma. For multidegrees $\alpha$ in $\NN^n$ and monomials $\xx^\alpha$, let $|\alpha|:=\sum_{i=1}^n \alpha_i$,
and similarly denote the $\NN$-degree of $\xx^\alpha$
by $|\xx^\alpha|:=|\alpha|$.

\begin{lemma}
\label{lem: unique-multidegrees} 
In the above setting, consider any multidegree $\gamma$ in $\NN^n$ occurring in $M''$ with $|\gamma| > r''$, so that $|\gamma|=r+r'+kd$ with $k \geq 1$.  Then the $\gamma$-homogeneous component of the multiplication
map $(M \otimes_R M')_\gamma \overset{\varphi_\gamma}{\rightarrow} M''_\gamma$ is a $\kk$-module isomorphism.  In other words, one has
$
\xx^\alpha \otimes_R \xx^{\alpha'}
=\xx^\beta \otimes_R \xx^{\beta'}
$
whenever $\alpha+\alpha'=\gamma=\beta+\beta'$.
\end{lemma}

%The idea is that if $\xx^\alpha \otimes_R \xx^{\alpha'}$ has high enough $\NN$-degree, one can push monomials of degree $d$ back and forth across $\otimes_R$ to obtain any other element with that multidegree.

\begin{proof}
Let $\aa,\aa',\bb,\bb'$ be monomials with
$\aa\otimes_R \aa'$ and $\bb \otimes_R \bb'$ in
$
\left( 
M \otimes_R M'
%\vermod{d}{ r}\otimes_R \vermod{d}{r'}
\right)_\gamma,
$
so that
$
\aa \cdot \aa' = \xx^\gamma = \bb \cdot \bb',
$
and assume $|\gamma|=(r+r')+kd$ with $k\geq 1$.
We wish to show that 
$\aa\otimes_R \aa'=\bb \otimes_R \bb'$.
By moving elements of $R=S^{(d)}$ across the tensor symbol, one may assume without loss of generality
that $|\aa|=|\bb|=r$, 
so $|\aa'|=|\bb'|=r'+kd$.  Also without loss of generality, assume $r\leq r'$.

Let $\cc = \gcd(\aa,\bb)$, and write 
$$
\begin{aligned}
\aa &= \cc\cdot\cc',\\
\bb &= \cc \cdot \cc''.
\end{aligned}
$$
Then $
\aa \cdot \aa' = \bb \cdot \bb'
$
implies 
$\cc'$ divides $\bb'$.  Since $\abs{\cc'}\leq \abs{\aa}=r \leq r'$, while $|\bb'|=r'+kd$, one can express
$\bb'=\cc' \cdot \pp \cdot \qq$
with $\abs{\cc' \cdot \pp} \equiv 0 \bmod{d}$, that is, with $\cc' \cdot \pp$ lying in $R$. Therefore
$$
\begin{aligned}
\bb \otimes_R \bb'
&= \cc \cdot \cc'' \otimes_R \cc' \cdot \pp \cdot \qq\\
&= \cc \cdot \cc'' \cdot \cc' \cdot \pp  \otimes_R \qq\\
&= \aa \cdot \cc'' \cdot \pp \otimes_R \qq\\
&= \aa  \otimes_R \cc'' \cdot \pp \cdot \qq \\
&= \aa \otimes_R \aa'.\qedhere
\end{aligned}
$$ 
\end{proof}

\begin{proof}[Proof of Theorem~\ref{thm: tensor-product} (i).]
Note that both $R$-modules $M \otimes_R M'$
and $M''$ vanish in $\NN$-degrees strictly below $r''=r+r'$.
Lemma~\ref{lem: unique-multidegrees} already shows that
the $GL(V)$-equivariant short exact sequence of $R$-modules
$$
0 \to \ker(\varphi) 
\to 
M \otimes_R M' 
\to M'' \to 0
$$
has $\ker(\varphi)$ vanishing in $\NN$-degrees strictly above
$r''$.  
On the other hand, restricting the sequence to degree exactly $r''$
gives this short exact sequence of $GL(V)$-modules
$$
0 \to \ker(\varphi)_{r''}
\to S^r \otimes_\kk S^{r'} 
\to S^{r+r'}
\to 0.
$$
This shows that 
$\ker(\varphi)_{r''} \cong \bbs^{\sigma(r,r')}(-r'')$
as $GL(V)$-representations, by comparing it to the special case of 
Proposition~\ref{concatenation-near-concatenation-prop}
with  $\alpha=(r),\beta=(r')$.

Lastly, since $\varphi$ is an $R$-module map,
$\ker(\varphi)$ must be an $R$-submodule, and since it is concentrated
in $\NN$-degree $r''$, it must be annihilated by $R_+$.
\end{proof}

Lemma~\ref{lem: unique-multidegrees}
also leads to a somewhat surprising family of 
$R$-module splittings for the multiplication map $M \otimes_R M' \overset{\varphi}{\twoheadrightarrow} M''$.
Order multidegrees $\alpha,\beta$ in $\NN^n$
componentwise by $\beta \leq \alpha$
if $\beta_i \leq \alpha_i$ for $i=1,2,\ldots,n$,
and choose for each $\alpha$ in $\NN^n$
a family of scalars 
$\{ c_{\alpha,\beta}: \beta \leq \alpha \text{ and } |\beta|=r \}$ satisfying this requirement:
\begin{equation}
\label{section-required-equality}
\sum_{\substack{\beta: \beta \leq \alpha,\\|\beta|=r}} 
c_{\alpha,\beta}=1.
\end{equation}
For example, a choice of such
scalars arises by fixing any total ordering $\prec$ on $\NN^n$ and
letting 
$$
c_{\alpha,\beta}=
\begin{cases} 1 & \text{ if }
\beta\text{ is the }\prec\text{-minimum of }\{ \beta': \beta' \leq \alpha \text{ and }|\beta'|=r\},\\
0 & \text{ otherwise}.
\end{cases}
$$

The next lemma then immediately applies assertion (ii) of
Theorem~\ref{thm: tensor-product}.

\begin{lemma}
\label{lem: product-map-split-surjection} 
For any choice of $\{ c_{\alpha,\beta} \}$ as above,
the $\kk$-module map $\psi: M'' \rightarrow M \otimes_R M'$ defined for $\alpha$ in $\NN^n$
with $|\alpha|=r+r'$ by
$$
\psi(\xx^\alpha):=  \sum_{\substack{\beta \leq \alpha\\ |\beta|=r}} c_{\alpha,\beta} \,\, \xx^\beta \otimes_R \xx^{\alpha-\beta}
$$
gives an $R$-module splitting of the multiplication
map 
$
M \otimes_R M' \overset{\varphi}{\twoheadrightarrow} M''.
%\varphi:\vermod{d}{ r}\otimes_{\veralg{d}} \vermod{d}{r'} \ra \vermod{d}{r+r'}
$
\end{lemma}

\begin{proof}
One has $\varphi \circ \psi=1$ since
$$
\varphi(\psi(\xx^\alpha))
=\sum_{\substack{\beta \leq \alpha\\ |\beta|=r}} c_{\alpha,\beta} \xx^\beta \cdot \xx^{\alpha-\beta}
= \xx^\alpha \sum_{\substack{\beta \leq \alpha\\ |\beta|=r}} c_{\alpha,\beta}
=\xx^\alpha.
$$
What is perhaps surprising is that $\psi$ is  $R$-linear, which one checks as follows.  Using the $\kk$-linearity of $\psi$, it suffices to show for monomials $\xx^\alpha, \xx^\gamma$ in $M'',R$,  that one has an equality
$$
    \psi(\xx^{\gamma} \xx^\alpha )
    =\xx^\gamma \psi(\xx^\alpha). 
$$
If $\xx^\gamma=1$ this is vacuously true.
If $\xx^\gamma \neq 1$, then $|\gamma+\alpha|>r''$, and
the desired equality is
$$
\sum_{\substack{\beta' \leq \gamma+\alpha\\ |\beta'|=r}} c_{\gamma+\alpha,\beta'} \,\, \xx^{\beta'} \otimes_R \xx^{\gamma+\alpha-\beta'}
=
\sum_{\substack{\beta \leq \alpha\\ |\beta|=r}} c_{\alpha,\beta} \,\, \xx^{\gamma+\beta} \otimes_R \xx^{\alpha-\beta}
$$
which follows from
Lemma \ref{lem: unique-multidegrees}
together with \eqref{section-required-equality} for
$\gamma+\alpha$ and $\alpha$.
\end{proof}

Assertion (iii) of Theorem~\ref{thm: tensor-product}
follows from our next lemma, the last of this section.

\begin{lemma}
When $\binom{r+r'}{r}$ lies in $\kk^\times$, 
the $R$-module splitting
$
\psi: M'' \longrightarrow M \otimes_R M'
$
can be chosen to be $GL(V)$-equivariant.
\end{lemma}
\begin{proof}
We will define $\psi$ on monomials
$\xx^\alpha$ in $M''=\vermod{d}{r''}$, and extend $\kk$-linearly. 

\vskip.1in
\noindent
{\sf Case 1.} $|\alpha|> r''$.
Define 
$\psi(\xx^\alpha):=\xx^\beta \otimes_R \xx^{\alpha-\beta}$, where 
$\beta \leq \alpha$ has $|\beta|=r$ and is otherwise arbitrary, e.g., $\beta$ is the $\prec$-minimum of $\{\beta':\beta' \leq \alpha \text{ and }|\beta|=r\}$ for a total order $\prec$ on $\NN^n$.

\vskip.1in
\noindent
{\sf Case 2.} $|\alpha|= r''$.
Define 
\begin{equation}   \label{natural-choice-of-section-constants} 
\psi(\xx^\alpha)=
\frac{1}{\binom{r+r'}{r}}
\sum_{\substack{\beta \leq \alpha:\\ |\beta|=r}} \binom{\alpha_1}{\beta_1} \binom{\alpha_2}{\beta_2} \cdots \binom{\alpha_n}{\beta_n} \xx^\beta \otimes_R \xx^{\alpha-\beta}.
\end{equation}

One readily checks that condition \eqref{section-required-equality} holds
in Case 1. For Case 2, use the identity 
$$
\binom{r+r'}{r}=\sum_{\substack{\beta \leq \alpha:\\ |\beta|=r}}\binom{\alpha_1}{\beta_1} \binom{\alpha_2}{\beta_2} \cdots \binom{\alpha_n}{\beta_n}
$$
coming from expanding both sides of 
$(1+t)^{r+r'} = (1+t)^{\alpha_1}(1+t)^{\alpha_2} \cdots (1+t)^{\alpha_n}$
binomially, and comparing coefficients of $t^r$.
Thus by Lemma~\ref{lem: product-map-split-surjection}, this $\psi$ defines an $R$-module splitting.

To check that this splitting $\psi$ is also $GL(V)$-equivariant, one
can check it separately in each $\NN$-degree $m \geq r''$.
We distinguish the same two cases as before.
\vskip.1in
\noindent
{\sf Case 1.}  $m > r''$.
It suffices to check for monomials $\xx^\alpha$ with
$|\alpha|>r''$ and $g$ in $GL(V)$, that
$$
g(\psi(\xx^\alpha))=\psi(g(\xx^\alpha)).
$$
Here $\psi(\xx^\alpha)=\xx^\beta \otimes_R \xx^{\alpha-\beta}$
for some $\beta$ with $|\beta|=r$.  Naming the coefficients $c_\gamma,d_\delta$
in $\kk$ appearing in these unique expansions
$$
\begin{aligned}
g(\xx^\beta)&=\sum_{|\gamma|=r} c_\gamma \xx^\gamma,\\
g(\xx^{\alpha-\beta})&=\sum_{|\delta|=r'} d_\delta \xx^\delta,
\end{aligned}
$$
one finds that
\begin{equation}
    \label{g-psi-summation}
g(\psi(\xx^\alpha)) 
=g(\xx^\beta \otimes_R \xx^{\alpha-\beta})
=g(\xx^\beta) \otimes_R g(\xx^{\alpha-\beta})
= \sum_{|\epsilon|=r''} 
\sum_{\substack{\gamma+\delta=\epsilon\\|\gamma|=r\\|\delta|=r'}}
c_\gamma d_\delta \,\, \xx^\gamma \otimes_R \xx^\delta.
\end{equation}
On the other hand, one also finds that
$$
g(\xx^\alpha) =
g(\xx^\beta \cdot \xx^{\alpha-\beta})
=g(\xx^\beta) \cdot g(\xx^{\alpha-\beta})
=\sum_{|\epsilon|=r''} 
\left( \sum_{\substack{\gamma+\delta=\epsilon\\|\gamma|=r\\|\delta|=r'}}
c_\gamma d_\delta \right) \xx^\epsilon
$$
and therefore 
\begin{equation}
    \label{psi-g-summation}
\psi(g(\xx^\alpha))
= \sum_{|\epsilon|=r''}
\left( \sum_{\substack{\gamma+\delta=\epsilon\\|\gamma|=r\\|\delta|=r'}}
c_\gamma d_\delta \right) \psi(\xx^\epsilon).
\end{equation}
Lemma~\ref{lem: unique-multidegrees} shows that
\eqref{g-psi-summation} and \eqref{psi-g-summation} are equal.

\vskip.1in
\noindent
{\sf Case 2.}  $m = r''$.
One can check that the formula \eqref{natural-choice-of-section-constants}
was chosen so as to make the splitting $\psi$ have its homogeneous component 
$
\psi_{r''}: M''_{r''} \longrightarrow (M \otimes_R M')_{r''}
$
equal to $\frac{1}{\binom{r+r'}{r}}$ times the following composite of
$GL(V)$-equivariant maps:
$$
\begin{array}{cccccccc}
S^{r+r'} &\hookrightarrow & S &\overset{\Delta}{\longrightarrow}
& S \otimes_\kk S &\twoheadrightarrow &S^r \otimes_\kk S^{r'}\\
& & \Vert &  &   \Vert& &\\
 & & \displaystyle\bigoplus_d S^d &  &  \displaystyle \bigoplus_{i,j} S^i \otimes_\kk S^{j}& &
\end{array}
$$
Here $\Delta$ in the middle is the comultiplication in
the usual bialgebra structure on $S=S(V)$, defined on the
algebra generators $x_i$ in $V$ via $\Delta(x_i):=1 \otimes_\kk x_i + x_i \otimes_\kk 1$.
\end{proof}

\begin{example}
When $d=2$ and $r=r'=1$, one has $R=S^{(2)}$
and $M=M'=\vermod{2}{1}$,  $M''=\vermod{2}{2}=S^{(2)}_+$, with $\bbs^{\sigma(1,1)}=\wedge^2$.  Here the exact sequence of Theorem~\ref{thm: tensor-product} is
$$
0 \to \wedge^2(-2) \to 
\vermod{2}{1} \otimes_{\veralg{2}} \vermod{2}{1} \to S^{(2)}_+ \to 0,
$$
which for any $\kk$ is an exact sequence of
$GL(V)$-representations, and a split exact sequence of $\veralg{2}$-modules.  However, if $\kk$ has characteristic $2$ and $\dim(V)\geq 2$, it does
not split as $GL(V)$-representations
in its $\NN$-degree $2$, where
it gives this well-known non-split sequence:
$$
0 \to \wedge^2(V) \to 
V \otimes_\kk V \to 
S^2(V) \to 0.
$$
\end{example}

%%%%%%%%%%%%%%
\subsection{Higher Tor and proof of Theorem~
\ref{higher-tor-is-all-socle-and-Artinian}.}
\label{tor-between-the-modules-subsection}
We recall the statement of the theorem.

\vskip.1in
\noindent
{\bf Theorem~\ref{higher-tor-is-all-socle-and-Artinian}.}
{\it
Fix $d,r,r' \geq 1$,
and let $R,M,M'$ denote
$\veralg{d}, \vermod{d}{ r}, \vermod{d}{ r'}$, as usual.
Then for $i \geq 1$, the $R$-module $\Tor^R_i(M,M')$ is annihilated by $R_+$, and as a module over $\kk=R/R_+$, has a $GL(V)$-isomorphism
$$
\Tor^R_i(M,M')
\cong
\schurfunctor^{\sigma(r,d^i,r')}.
$$
}
\vskip.1in

\begin{proof}
Resolve $M'$ over $R$ as in Theorem~\ref{skew-Schur-functor-theorem}, and apply $M \otimes_R (-)$ to give a complex
$$
0 
\leftarrow M \otimes_\kk \schurfunctor^{(r)}
\leftarrow M \otimes_\kk \schurfunctor^{\sigma(d,r')}
\leftarrow M \otimes_\kk \schurfunctor^{\sigma(d,d,r')}
%\leftarrow M \otimes_\kk \schurfunctor^{\sigma(d,d,d,r')}
\leftarrow \cdots
$$
whose homology computes $\Tor^R_i(M,M')$.
The $i^{th}$ term $M \otimes_\kk \bbs^{\sigma(d^i,r')}$ vanishes in degrees
below $di+r+r'$.  For $m \geq 0$, the
degree ${d(i+m)+r+r'}$ component of the 
boundary map
$$
\begin{array}{ccc}
(
M \otimes_\kk \bbs^{\sigma(d^i,r')}
)_{d(i+m)+r+r'}&
%\overset{(\partial_i)_{d(i+m)+r+r'}}{\longrightarrow}
\xrightarrow[]{(\partial_i)_{d(i+m)+r+r'}}
&
(
M \otimes_\kk \bbs^{\sigma(d^{i-1},r')}
)_{d(i+m)+r+r'}\\
\Vert& &\Vert \\
S^{dm+r} \otimes_\kk
\bbs^{\sigma(d^i,r')} & &
S^{d(m+1)+r} \otimes_\kk
\bbs^{\sigma(d^{i-1},r')}
\end{array}
$$
may be identified with the component
$\partial^{(d^i,r')}_{d(i+m)+r+r'}$
within the complex of ribbons.

In particular, if $m=0$, 
then Lemma \ref{lem:ribbonCxLemma}(i) gives
a $GL(V)$-isomorphism
$$
\ker \partial^{(d^i,r)}_{di + r + r'} \cong \bbs^{\sigma(r,d^i,r')} 
$$
and for $m \geq 1$, Lemma \ref{lem:ribbonCxLemma}(ii)
shows that 
$$
\ker \partial^{(d^i,r')}_{d(i+m)+r+r'} = \im \partial^{(d^{i+1},r')}_{d(i+m)+r+r'}.
$$
Thus $\Tor_i^R(M,M')
\cong \bbs^{\sigma(r,d^i,r')} $, which is
pure of degree $di+r+r'$ and annihilated by $R_+$.
\end{proof}

%%%%%%%%%%%%%%
\subsection{Hom and proof of Theorem~\ref{ext-zero-theorem}.}
\label{hom-section}
Recall the statement of the theorem.

\vskip.1in
\noindent
{\bf Theorem~\ref{ext-zero-theorem}.}
{\it
Assume $n \geq 2$, so that $S=\kk[x_1,\ldots,x_n]$
is not univariate.  Fix an integer $d \geq 1$, 
% $\kk^\times$, 
defining $R=\veralg{d}$.
For $r, r' \geq 0$,
consider three $R$-modules 
$$
\begin{aligned}
M&=\vermod{d}{ r},\\
M'&=\vermod{d}{ r'}, \\
M''&=\vermod{d}{ r''},
\end{aligned}
$$
defining
%in $(r'-r) +d\ZZ$ 
$r'':= r'-r$ if $r \leq r'$,
otherwise if $r>r'$, defining $r''$ to be the unique integer in $[0,d)$ congruent to $r'-r \bmod{d}$.
Then one has a $GL(V)$-equivariant $R$-module isomorphism 
$$
\begin{array}{rcl}
M'' &\longrightarrow &\hom_R (M , M' )\\
m'' &\longmapsto & (m \mapsto m''\cdot m).
%\quad \left( = \Ext^0_R(M,M') \right)
\end{array}
$$
}
\vskip.1in

Note that this map sending $m''$ to multiplication by $m''$ 
is $GL(V)$-equivariant and an $R$-module
map, and it is always {\it injective}.
What is not obvious is its {\it surjectivity}.
We build up the proof of
the theorem with a few lemmas.

\begin{lemma}\label{lem: tensorIsoHom}
For any integers $r,r',r'' \geq 0$, there is an isomorphism
$$\hom_R (\vermod{d}{r+r'} , \vermod{d}{r''}) \cong \hom_R (\vermod{d}{r} \otimes_R \vermod{d}{r'} , \vermod{d}{r''}).$$
\end{lemma}

\begin{proof}
Apply $\hom_R ( - , \vermod{d}{r''})$ to the split exact sequence from Theorem~\ref{thm: tensor-product}(i)
$$0 \to \bbs^{\sigma (r,r')}  \to \vermod{d}{r} \otimes_R \vermod{d}{r'} \to \vermod{d}{r+r'} \to 0$$
and note $\hom_R (\bbs^{\sigma (r,r')}  , \vermod{d}{r''}) = 0$ because  $\bbs^{\sigma (r,r')} $ is all $R$-torsion.
\end{proof}

\begin{lemma}\label{lem: grade2} 
For $n:=\dim V \geq 2$, with $0 \leq r < d$ and $i \geq 1$, one has
%the sequence $$x_1^{di} , x_2^{di} , \dots , x_n^{di}$$ is regular on $\vermod{d}{r}$. In particular,
$$\ext^1_R (\veralg{d} / \vermod{d}{di} , \vermod{d}{r} ) = 0.$$
\end{lemma}

\begin{proof}
Note $R=\veralg{d}$ has
 $I:=\vermod{d}{di}=(R_+)^i$, the $i^{th}$ power
 of the graded maximal ideal $R_+$.
 Denoting $M:=\vermod{d}{r}$, it suffices to show that
 whenever $0 \leq r < d$ we have
 that $\ext^j_R(R/I,M)=0$ for $0 \leq j \leq n-1$.
 This vanishing is controlled by
 the {\it grade} of $I$ on $M$, denoted $\mathrm{grade}(I,M)$, defined as the maximal
 length $\ell$ of an $M$-regular sequence $(\theta_1,\ldots,\theta_\ell)$ contained in $I$:
 one has (see, e.g., Bruns and Herzog \cite[Thm. 1.2.5]{BrunsHerzog}),  
 $$
 \mathrm{grade}(I,M):=
 \min\{ \ell: \Ext_R^\ell(R/(I, M) \neq 0\}.
 $$
 Now $S$ is a Cohen-Macaulay $R$-module,
 and an $R$-module direct sum
 $
 S=\bigoplus_{r=0}^{d-1} \vermod{r}{d}.
 $
 Each $R$-summand $M=\vermod{d}{r}$
 is therefore also a Cohen-Macaulay $R$-module.
 So $x_1^{di},\ldots,x_n^{di}$ in $I =(R_+)^i$
 give an $M$-regular sequence of length $n$,
 showing $\mathrm{grade}(I,M)=n$.
%Notice that $x_1^{di} , \dots , x_n^{di}$ is regular on the polynomial ring $S$, and $S$ decomposes as a direct sum of $\veralg{d}$-modules:
%$$S = \bigoplus_{r=0}^{d-1} \vermod{d}{r}.$$
%It follows that if $f \in \vermod{d}{r}$ and 
%$$x_j^{di} f \in (x_1^{di} , \dots , x_{j-1}^{di} ) \vermod{d}{r} \subset  (x_1^{di} , \dots , x_{j-1}^{di} ) S,$$
%for some $j \leq n$, then regularity over $S$ implies that $f \in (x_1^{di} , \dots , x_{j-1}^{di}) S$. Since $f$ has degree $\equiv r \mod d$ and each $x_\ell^{di}$ has degree $\equiv 0 \mod d$, a degree count shows that $f \in (x_1^{di} , \dots , x_{j-1}^{di}) \vermod{d}{r}$. In particular, this shows that the grade of $\veralgplus{d}^i$ on $\vermod{d}{r}$ is at least $n$. On the other hand, this grade is precisely $\min\{ j: \Ext_R^j(R/(\vermod{d}{di}, M) \neq 0\}$; see, e.g., Bruns and Herzog \cite[Thm. 1.2.5]{BrunsHerzog}), so the latter statement follows.
\end{proof}

\begin{remark}
The hypothesis $0 \leq r < d$ in Lemma~\ref{lem: grade2}
is important. When $r \geq d$, the $R$-modules
$\vermod{d}{r}$ are {\it not} $R$-summands of $S$, and are 
{\it not} Cohen-Macaulay for $n \geq 2$.  They have $R$-module
depth $1$ for $r \geq d$, as observed by Greco and Martino in \cite[Thm. 3.5]{GrecoMartino}.
\end{remark}

The next lemma is the case $r=r'$ in Theorem \ref{ext-zero-theorem}.
\begin{lemma}\label{lem: Hom(r,r) is R}
The inclusion $R \hookrightarrow \hom_R (\vermod{d}{r} , \vermod{d}{r})$ is an isomorphism for all $r \geq 0$.
\end{lemma}

\begin{proof}
We prove this first for $r \equiv 0 \bmod{d}$, and then induct downward for $r \not\equiv 0 \bmod{d}$. 

Fix any $i \geq 1$ with $r=di$, and apply $\hom_R (-, \vermod{d}{di})$ to the short exact sequence
\begin{equation}\label{eq: ses-di}
0 \to \vermod{d}{di} \to R \to Q \to 0.  
\end{equation}
where $R:=\veralg{d}$ and  $Q:=\veralg{d}/\vermod{d}{di}$,
giving a long exact sequence in $\ext$.
Its first term vanishes
$
\hom_R (Q , \vermod{d}{di}) = 0
$
since $Q$ is all $R$-torsion. 
As $\ext_R^1(R,\vermod{d}{di})=0$, one has a short exact sequence
$$0 \to \hom_R (R , \vermod{d}{di}) \to \hom_R (\vermod{d}{di} , \vermod{d}{di}) \to \ext^1_R (Q , \vermod{d}{di}) \to 0.$$
On the other hand, applying  $\hom_R (Q , -) $ to the sequence (\ref{eq: ses-di}) yields 
\begin{equation}
\label{second-long-exact}
\begin{aligned}
0\ra \Hom_R(Q,\vermod{d}{di}) \ra \Hom_R(Q,R)\ra \Hom_R(Q,Q) \\
\ra \Ext^1_R(Q,\vermod{d}{di}) \ra \Ext^1_R(Q,R)\ra \dots
\end{aligned}
\end{equation}
but again $\hom_R(Q,\vermod{d}{di}) = \hom_R(Q,R) = 0$, as $Q$ is  $R$-torsion.  Also $\hom_R(Q,Q) \cong Q$ since $Q$ is a cyclic $R$-module.
Lastly, $\Ext^1_R(Q,R) = 0$ by the $r=0$ case of Lemma~\ref{lem: grade2}.  Thus \eqref{second-long-exact} yields
an isomorphism
$
Q \cong \ext^1_R (Q , \vermod{d}{di}).
$
One can check that this lets one assemble a morphism of 
short exact sequences given by the following commuting diagram:
% https://q.uiver.app/?q=WzAsMTAsWzAsMCwiMCJdLFsxLDAsIlxcaG9tX1IgKFxcdmVyYWxne2R9ICwgXFx2ZXJhbGdwbHVze2R9KSJdLFsyLDAsIlxcaG9tX1IgKFxcdmVyYWxncGx1c3tkfSAsIFxcdmVyYWxncGx1c3tkfSkiXSxbMywwLCJcXGV4dF4xX1IgKFxca2sgLCBcXHZlcmFsZ3BsdXN7ZH0pIl0sWzQsMCwiMCJdLFswLDEsIjAiXSxbMSwxLCJcXHZlcmFsZ3BsdXN7ZH0iXSxbMiwxLCJcXHZlcmFsZ3tkfSJdLFszLDEsIlxca2siXSxbNCwxLCIwIl0sWzgsOV0sWzMsNF0sWzIsM10sWzEsMl0sWzAsMV0sWzUsNl0sWzYsMSwiXFxzaW0iXSxbNywyLCIiLDAseyJzdHlsZSI6eyJ0YWlsIjp7Im5hbWUiOiJob29rIiwic2lkZSI6InRvcCJ9fX1dLFs4LDMsIlxcc2ltIl0sWzcsOF0sWzYsN11d
\[\begin{tikzcd}
	0 & {\hom_R (R , \vermod{d}{di})} & {\hom_R (\vermod{d}{di} , \vermod{d}{di})} & {\ext^1_R (Q , \vermod{d}{di})} & 0 \\
	0 & {\vermod{d}{di}} & {R} & Q & 0
	\arrow[from=2-4, to=2-5]
	\arrow[from=1-4, to=1-5]
	\arrow[from=1-3, to=1-4]
	\arrow[from=1-2, to=1-3]
	\arrow[from=1-1, to=1-2]
	\arrow[from=2-1, to=2-2]
	\arrow["\sim", from=2-2, to=1-2]
	\arrow[hook, from=2-3, to=1-3]
	\arrow["\sim", from=2-4, to=1-4]
	\arrow[from=2-3, to=2-4]
	\arrow[from=2-2, to=2-3]
\end{tikzcd}\]
The outer two vertical maps are isomorphisms, so the snake lemma implies the middle map is an isomorphism, proving the
$r \equiv 0\bmod{d}$ case.

In the case $d(i-1) < r < di$ we proceed by downward induction on $r$.  Apply the (left-exact) functor $\hom_R (\vermod{d}{r} , -)$ to the inclusion $\vermod{d}{r} \hookrightarrow \hom_R (\vermod{d}{1} , \vermod{d}{r+1})$. This induces an inclusion:
\begin{align*}
\hom_R (\vermod{d}{r} , \vermod{d}{r}) \hookrightarrow &\hom_R (\vermod{d}{r} , \hom_R (\vermod{d}{1} , \vermod{d}{r+1})) \\ &\cong \hom_R (\vermod{d}{r} \otimes_R \vermod{d}{1},  \vermod{d}{r+1}),
\end{align*}
where the latter isomorphism is Hom-Tensor adjunction. 
Lemma \ref{lem: tensorIsoHom} gives an isomorphism 
$$\hom_R (\vermod{d}{r} \otimes_R \vermod{d}{1},  \vermod{d}{r+1}) \cong \hom_R (\vermod{d}{r+1} , \vermod{d}{r+1} ),
$$ 
and by induction, $\hom_R (\vermod{d}{r+1} , \vermod{d}{r+1} ) \cong \veralg{d}$. It follows that there are inclusions
$$\veralg{d} \hookrightarrow \hom_R (\vermod{d}{r} , \vermod{d}{r}) \hookrightarrow \veralg{d},$$
and one can check that composing these two inclusions is the identity. 
Thus the inclusion $$\hom_R (\vermod{d}{r} , \vermod{d}{r}) \hookrightarrow \veralg{d}$$ has a right inverse and is hence also a surjection, so the isomorphism follows.
\end{proof}

The next lemma is the case $r'=0$ in Theorem \ref{ext-zero-theorem}.
\begin{lemma}
\label{lem:claim3}
For $\dim V \geq 2$ and any $r \geq 0$, one has that
% $$\ext^1_R (\veralg{d}/\vermod{d}{di} , \vermod{d}{di-r} ) = 0 $$
% and 
$$
\hom_R (\vermod{d}{r} , \veralg{d}) \cong \vermod{d}{di-r},
$$
where $i$ is the smallest integer such that $r \leq di$.
\end{lemma}

\begin{proof}
Observe first that
\begingroup\allowdisplaybreaks
\begin{align*}
    \hom_R ( \vermod{d}{r} , \veralg{d}) 
    &\cong \hom_R (\vermod{d}{r} , \hom_R (\vermod{d}{di-r} , \vermod{d}{di-r} ) ) \quad &(\textrm{Lemma~\ref{lem: Hom(r,r) is R}}) \\
    &\cong \hom_R (\vermod{d}{r} \otimes_R \vermod{d}{di-r} , \vermod{d}{di-r} ) \quad &(\textrm{adjunction}) \\
    &\cong\hom_R (\vermod{d}{di} , \vermod{d}{di-r} ) \quad &(\textrm{Lemma~\ref{lem: tensorIsoHom}}).
\end{align*}
\endgroup
Letting $Q:=\veralg{d}/\vermod{d}{di}$ again, and applying $\hom_R (- , \vermod{d}{di-r})$ to the exact sequence
$$
0 \to \vermod{d}{di} \to R \to Q \to 0
$$
yields, as in the proof of Lemma~\ref{lem: Hom(r,r) is R}, the short exact sequence
$$0 \to \vermod{d}{di-r} \to \hom_R (\vermod{d}{di} , \vermod{d}{di-r} ) \to \ext^1_R (Q , \vermod{d}{di-r} ) \to 0.$$
One can apply Lemma~\ref{lem: grade2} to conclude $\ext^1_R (Q , \vermod{d}{di-r} ) = 0$, since the assumption on $i$ implies that $0 \leq di - r < d$.  Thus $\hom_R (\vermod{d}{di} , \vermod{d}{di-r} ) \cong \vermod{d}{di-r}$.
\end{proof}

% % \begin{lemma}\label{lem:theBaseCase}
% % There is an isomorphism
% % $$\hom_R (\vermod{d}{1} , \vermod{d}{r} ) = \begin{cases}
% % \vermod{d}{r-1} & \textrm{if} \ r \geq 1, \\
% % \vermod{d}{d-1} & \textrm{if} \ r=0.
% % \end{cases}$$
% % \end{lemma}

% % \begin{proof}
% % % In view of the final part of the above proof, it suffices to prove that there is an isomorphism
% % % $$\hom_R (\vermod{d}{1} , \vermod{d}{r}) \cong \vermod{d}{r-1}$$
% % % for all $r \geq 1$. It is of no loss of generality to assume that $r \geq 1$ since applying $\hom_R (\vermod{d}{1} , -)$ to the short exact sequence
% % % $$0 \to \veralgplus{d} \to \veralg{d} \to \kk \to 0$$
% % % implies that $\hom_R (\vermod{d}{1} , \veralg{d}) = \hom_R (\vermod{d}{1} , \veralgplus{d})$.
% % % We will first need to establish a few claims:

% % To conclude the proof, let $r \geq 1$ be any integer and choose $i$ to be the smallest integer such that $r \leq di$. The above claims yield a string of isomorphisms:
% % \begingroup\allowdisplaybreaks
% % \begin{align*}
% %     \hom_R (\vermod{d}{1} , \vermod{d}{r} ) &= \hom_R (\vermod{d}{1} , \hom_R (\vermod{d}{di-r} , \veralg{d} ) ) \quad (\textrm{Claim 3}) \\
% %     &= \hom_R (\vermod{d}{1} \otimes_R \vermod{d}{di-r} , \veralg{d} ) \quad (\textrm{adjunction}) \\
% %     &= \hom_R (\vermod{d}{di-r+1} , \veralg{d}) \quad (\textrm{Claim 1}) \\
% %     &= \vermod{d}{r-1} \quad (\textrm{Claim 3}).
% % \end{align*}
% % \endgroup
% % \end{proof}

\begin{proof}[Proof of Theorem \ref{ext-zero-theorem}]
We induct on $r$, with
base cases $r=0,1$ handled separately.

\vskip.1in
\noindent
{\sf The base case $r=0$.} There is not much to
prove, as $r=0$ implies $M=R$ and $r''=r'$ so  $M''=M'$.  Here the theorem says the inclusion $M' \hookrightarrow \hom_R(R,M')$ is an isomorphism.

\vskip.1in
\noindent
{\sf The base case $r=1$.}
When $r=1$ and $r'=0$, the theorem is an instance of Lemma~\ref{lem:claim3}. When $r=1$ and $r'=1$, the theorem is an instance of 
Lemma~\ref{lem: Hom(r,r) is R}.
Consequently, one may assume without
loss of generality that $r' \geq 2$. Choose $i$ to be the smallest integer such that $r' \leq di$. The preceding lemmas yield a string of isomorphisms:
%\begingroup\allowdisplaybreaks
\begin{align*}
    \hom_R (\vermod{d}{1} , \vermod{d}{r'} ) &= \hom_R (\vermod{d}{1} , \hom_R (\vermod{d}{di-r'} , \veralg{d} ) ) \quad &(\textrm{Lemma~\ref{lem:claim3}}) \\
    &\cong \hom_R (\vermod{d}{1} \otimes_R \vermod{d}{di-r'} , \veralg{d} ) \quad &(\textrm{adjunction}) \\
    &\cong \hom_R (\vermod{d}{di-r'+1} , \veralg{d}) \quad &(\textrm{Lemma}~\ref{lem: tensorIsoHom}) \\
    &\cong \vermod{d}{r'-1} \quad &(\textrm{Lemma~\ref{lem:claim3}}).
\end{align*}

\vskip.1in
\noindent
{\sf The inductive step where $r\geq 2$.}
When $r'=0$, the theorem is an instance of Lemma \ref{lem:claim3}, so assume without loss of generality that $r' \geq 1$.  There is then a string of isomorphisms:
\begin{align*}
&\hom_R (\vermod{d}{r} , \vermod{d}{r'} ) & \\
&= \Hom_R(\vermod{d}{r-1}\otimes_R \vermod{d}{1}, \vermod{d}{r'}) &\textrm{(Lemma~\ref{lem: tensorIsoHom})} \\
&\cong \Hom_R(\vermod{d}{r-1}, \Hom_R(\vermod{d}{1}, \vermod{d}{r'})) &\text{(Hom-tensor adjunction)} \\
&\cong \Hom_R(\vermod{d}{r-1}, \vermod{d}{r'-1}) & \textrm{(above $r=1$ base case)} \\
&\cong \vermod{d}{r''} &\text{(inductive hypothesis).}
\end{align*}
\end{proof}

\begin{remark}
Under some extra hypotheses on the scalars $\kk$, one can regard the special case where $0 \leq r,r' \leq d-1$ in Theorem~\ref{ext-zero-theorem}
as an instance of a more general statement,
Proposition~\ref{cor:abelian-auslander}
below.  Let $\kk$ be a field,
and $G$ a finite abelian subgroup of $GL_n(\kk)$ acting on $S=\kk[x_1,\ldots,x_n]$.  
For a {\it linear character} $\chi: G \rightarrow \kk^\times$
define the 
{\it $S^G$-module of relative invariants} in $S$ as follows:
$$
S^{G,\chi}:=\{f \in S: g(f)=\chi(g) f\}.
$$

\begin{prop}
\label{cor:abelian-auslander}
In the above setting with $\kk$ a field, assume furthermore that
\begin{itemize}
    \item 
$\#G$ lies in $\kk^\times$,
\item  $\kk$ contains the $e^{th}$ roots-of-unity, for $e:=\mathrm{lcm}\{ \text{multiplicative orders of }g\text{ in }G\}$, and
\item $G$ contains no pseudo-reflections,
that is, no $g$ with
$\ker(g - I_n)$ a hyperplane in $\kk^n$.
\end{itemize}
Then for any two linear characters $\chi, \chi': G \rightarrow \kk^\times$, defining a
third character $\chi'':=\chi^{-1}\chi'$, the following map
is an $S^G$-module isomorphism:
$$
\begin{array}{rcl}
S^{G,\chi''} &\longrightarrow&
\hom_{S^G}( S^{G,\chi}, S^{G,\chi'})\\
s'' &\longmapsto & (s \mapsto s''\cdot s).
\end{array}
$$
\end{prop}
The proof, which we omit here, uses the following result of
Auslander \cite[Thm. 5.15]{LeuschkeWiegand}
about the {\it skew group ring} $S \# G$,
which is defined to be a free $S$-module on $S$-basis $\{g\}_{g \in G}$, with multiplication defined
for $s,t \in S$ and $g,h \in G$ by
$$
(t \cdot h) \cdot (s \cdot g) := t \cdot h(s) \cdot hg.
$$

\begin{thm}[Auslander's theorem]\label{thm:auslander}
Let $\kk$ be a domain, 
and $G$ a finite subgroup of $GL_n(\kk)$ that contains no pseudo-reflections and has $\#G$ in $\kk^\times$.  Then the following map
is is an isomorphism of $S^G$-algebras, 
and hence also of $S^G$-modules:
$$
\begin{array}{rcl}
S \# G & \longrightarrow &\hom_{S^G} (S , S)\\
s \cdot g &\longmapsto& ( t \mapsto s \cdot g(t) ).
\end{array}
$$
\end{thm}
To recover the case $0 \leq r,r' \leq d-1$ in Theorem~\ref{ext-zero-theorem} from
Proposition~\ref{cor:abelian-auslander}, take
$G$ to be the cyclic subgroup of order $d$ given by all of the scalar matrices in $GL_n(\kk)$
having a $d^{th}$-root-of-unity on the diagonal.
When $n \geq 2$, this group $G$ contains no pseudoreflections.
\end{remark}

\begin{remark}
A natural next subject of study after computing $\hom_R (\vermod{d}{r} , \vermod{d}{r'})$ is on the behavior of the higher derived functors $\ext^R_i (\vermod{d}{r} , \vermod{d}{r'} )$ for $i > 0$. It turns out that this case is rather subtle in contrast to the case of $\tor$, and there is no similarly ``clean" statement as in Theorem \ref{ext-zero-theorem}.

There \emph{are} commonalities with the case of $\tor$ worth mentioning: it turns out that the higher $\ext$ modules are also annihilated by $R_+$, and the $\ext$-modules satisfy a more delicate persistence of nonvanishing than the $\tor$-modules. The tools and ideas used to prove these statements, along with some surprising connections between the Yoneda algebra and categorification of certain noncommutative Hopf algebras, will be the subject of a future paper.
\end{remark}

%%%%%%%%%%%%%%%%%%%%%%%%%%%%%%%%%%%%%%%%
\bibliographystyle{abbrv}
\bibliography{bibliography}

\end{document}